\documentclass[onefignum,onetabnum]{siamart220329}

\usepackage{lipsum}
\usepackage{amsfonts}
\usepackage{amssymb}
\usepackage{graphicx}
\usepackage{epstopdf}
\usepackage{algpseudocode}
\usepackage{tensor}
\usepackage{verbatim}
\usepackage{mathtools}
\usepackage{cleveref}
\usepackage{framed}
\usepackage{tikz}
\usetikzlibrary{calc, shapes, angles, quotes, patterns, decorations.pathreplacing, cd}
\usepackage{subcaption}
\usepackage{bm}
\usepackage{esint}
\usepackage{dsfont}

\ifpdf
\DeclareGraphicsExtensions{.eps,.pdf,.png,.jpg}
\else
\DeclareGraphicsExtensions{.eps}
\fi

\newlength{\leftstackrelawd}
\newlength{\leftstackrelbwd}
\def\leftstackrel#1#2{\settowidth{\leftstackrelawd}%
	{${{}^{#1}}$}\settowidth{\leftstackrelbwd}{$#2$}%
	\addtolength{\leftstackrelawd}{-\leftstackrelbwd}%
	\leavevmode\ifthenelse{\lengthtest{\leftstackrelawd>0pt}}%
	{\kern-.5\leftstackrelawd}{}\mathrel{\mathop{#2}\limits^{#1}}}

\newcommand{\bdd}[1]{ \boldsymbol{#1} }
\newcommand{\unitvec}[1]{\hat{\bdd{#1}}}
\newcommand{\vertiii}[1]{{\left\vert\kern-0.25ex\left\vert\kern-0.25ex\left\vert #1 
		\right\vert\kern-0.25ex\right\vert\kern-0.25ex\right\vert}}

\newsiamremark{remark}{Remark}
\newsiamremark{hypothesis}{Hypothesis}
\crefname{hypothesis}{Hypothesis}{Hypotheses}
\newsiamthm{claim}{Claim}

\headers{Statically Condensed Iterated Penalty Method}{M. Ainsworth and C. Parker}

\title{Statically Condensed Iterated Penalty Method for High Order Finite Element Discretizations of Incompressible Flow \thanks{Submitted to the editors DATE. \funding{The second author acknowledges that this material is based upon work supported by the National Science Foundation under Award No.
			DMS-2201487. }}}

\author{Mark Ainsworth\thanks{ Division of Applied Mathematics, Brown University, Providence, RI
		(\email{mark\_ainsworth@brown.edu}).} \and Charles Parker \thanks{ Mathematical Institute, University of Oxford, Andrew Wiles Building, Woodstock Road, Oxford OX2 6GG, UK  (\email{charles.parker@maths.ox.ac.uk})}  }

\usepackage{amsopn}

\DeclareMathOperator{\vcurl}{\mathbf{curl}}
\DeclareMathOperator{\grad}{\mathbf{grad}}
\DeclareMathOperator{\dive}{div}

\DeclareMathOperator{\supp}{supp}

\ifpdf
\hypersetup{
  pdftitle={SCIPforFEM},
  pdfauthor={M. Ainsworth and C. Parker}
}
\fi

\begin{document}

\maketitle

\begin{abstract}
	We introduce and analyze a Statically Condensed Iterated Penalty (SCIP) method for solving incompressible flow problems discretized with $p$th-order Scott-Vogelius elements. While the standard iterated penalty method is often the preferred algorithm for computing the discrete solution, it requires inverting a linear system with $\mathcal{O}(p^{d})$ unknowns at each iteration. The SCIP method reduces the size of this system to $\mathcal{O}(p^{d-1})$ unknowns while maintaining the geometric rate of convergence of the iterated penalty method. The application of SCIP to Kovasznay flow and Moffatt eddies shows good agreement with the theory.
\end{abstract}



\section{Introduction}
\label{sec:intro}

The search for stable mixed finite element pairs for the Stokes equations has a long and rich history and recently, attention has been focused on finite elements that satisfy exact sequence properties; see e.g. the review paper \cite{JohnLinkeMerdonNeilReb17} and references therein. Finite element spaces based on exact sequences are attractive in that they lead to schemes that exhibit ``pressure robustness" and result in approximations to the velocity that are pointwise divergence free. These schemes are often inf-sup stable with
respect to the mesh size and, in some cases, can be shown to be
\cite{AinCP19StokesI} uniformly stable with respect to the polynomial degree. Stability is crucial for avoiding nonphysical artifacts in the numerical solution, obtaining optimal a priori estimates, and constructing effective preconditioners.

A more classical approach to devising mixed finite element schemes,
particularly in the context of higher order methods, consists of using a combination of the form $\bdd{X}_D\times\dive\bdd{X}_D$ where the space $\bdd{X}_D$ consists of continuous piecewise polynomial vector fields.  Such schemes also form part of an exact sequence, but were not originally derived in this way~\cite{ScottVog84,ScottVog85,Vogelius83divinv}. While it is
known~\cite{ScottVog85,Vogelius83divinv} that these elements are inf-sup stable with respect to the mesh size provided that the space $\bdd{X}_D$ consists of fourth order polynomials or higher, the same
analysis~\cite{ScottVog85,Vogelius83divinv} suggested that the inf-sup constant may decay algebraically as the polynomial order is increased. However, practical experience suggests that the scheme is uniformly inf-sup stable in the polynomial degree; one by-product of the current work is \emph{a formal proof of the uniform inf-sup stability of the Scott-Vogelius elements in both the mesh size and the polynomial degree under certain necessary (but mild) assumptions on the mesh.} Despite providing the first proof of uniform stability, the main objective of the current work is quite different: we exhibit an algorithm that enables one to efficiently \emph{implement} the
Scott-Vogelius elements, particularly in the case of higher order elements.

One difficulty in applying the Scott-Vogelius elements is the difficulty of finding a basis for the discrete pressure space $\dive\bdd{X}_D$. In addition, as mentioned in \cref{sec:prelims}, the pressure space possesses non-trivial constraints at certain element vertices, which further exacerbates the problem. For these reasons, the method is often implemented using the Iterated Penalty
(IP) method \cite{Brenner08,FortinGlow83,Glowinski84,MorganScott18}. The IP approach circumvents the need to construct an explicit basis for the pressure space at the expense of proceeding iteratively which involves repeatedly having to solving finite element type problems involving only the velocity space $\bdd{X}_D$. This kind of approach is attractive in the context of lower order methods but, as remarked in \cref{sec:vanilla ip}, the standard Iterated Penalty approach becomes increasingly less attractive for higher order
elements owing to need to update the interior degrees of freedom on every iteration. 

It is worth noting that the interior degrees of freedom number
$\mathcal{O}(p^d)$ while the remaining degrees of freedom associated with element boundaries number $\mathcal{O}(p^{d-1})$. As such, the interior degrees of freedom account for the bulk of the degrees of freedom and having to update them at every iterate dominates the overall cost. To remedy this issue, we propose a Statically Condensed Iterated Penalty (SCIP) method which requires only the degrees of freedom on the element boundaries to be updated at each iteration; the result being that the cost per iteration of SCIP is drastically reduced compared with the standard iterated penalty method. 

Roughly speaking, the main idea behind the SCIP method consists of decomposing the discrete solution into contributions from a \emph{pair} of subspaces associated with element boundaries and from local pairs of subspaces associated with element interiors. Each of these contributions can be obtained by solving a Stokes-like equation posed over their respective subspace. The boundary contribution is first solved using the standard iterated penalty method, while
the interior contributions, for which bases may be readily constructed, are then solved via direct methods.  The net effect is that the SCIP method only requires a single solve for the interior degrees of freedom rather than having to update at every iteration using the IP method. 

In \cref{sec:modified it}, we provide theoretical bounds for the convergence of SCIP, and detail its implementation. In \cref{sec:numerics}, we present two numerical examples demonstrating SCIP's performance. \Cref{sec:stokes extension} introduces discrete extension operators that are then used to prove the convergence results of the SCIP method in \cref{sec:scip convergence}. Finally, \cref{sec:proof of optimal approx} contains the various properties of the Scott-Vogelius elements in 2D, including inf-sup stability, optimal approximation, and exact sequence properties.

\section{Mathematical Preliminaries}
\label{sec:prelims}

Let $\Omega \subset \mathbb{R}^d$, $d \in \{2,3\}$ be a polygonal domain whose boundary $\Gamma$ is partitioned into disjoint subsets $\Gamma_D$ and $\Gamma_N$ with $|\Gamma_D| > 0$. We consider the Stokes equations in  $\Omega$:
\begin{subequations}
	\label{eq:stokes continuous}
	\begin{alignat}{2}
	-\dive \bdd{\varepsilon}(\bdd{u}) + \grad q &= \bdd{f} \qquad && \text{in } \Omega, \\
	\label{eq:stokes div free}
	\dive \bdd{u} &= 0 \qquad & & \text{in } \Omega, \\
	\label{eq:le continuous dbc}
	\bdd{u} &= \bdd{0} \qquad && \text{on } \Gamma_D, \\
	\label{eq:le continuous nbc}
	-\bdd{\varepsilon}(\bdd{u})\cdot \unitvec{n} + q \unitvec{n} &= \bdd{g} \qquad &&\text{on } \Gamma_N,
	\end{alignat}
\end{subequations}
where $\bdd{u}$ and $q$ are the unknown fluid velocity  and pressure, $\bdd{\varepsilon}(\cdot)$ is the strain rate tensor, and $\bdd{f} \in \bdd{L}^2(\Omega)$ and $\bdd{g} \in \bdd{L}^2(\Gamma_N)$ are given data. 

Let $\bdd{H}_D^1(\Omega) :=  \{ \bdd{v} \in \bdd{H}^1(\Omega) : \bdd{v}|_{\Gamma_D} = \bdd{0} \}$ and $L^2_D(\Omega) = L^2(\Omega)$ if $|\Gamma_D| \neq |\Gamma|$ and $L^2_D(\Omega) = L^2_0(\Omega)$ otherwise. The variational form of \cref{eq:stokes continuous} is then: Find $(\bdd{u}, q) \in \bdd{H}^1_D(\Omega) \times L^2_D(\Omega)$ such that 
\begin{subequations}
\label{eq:stokes variational form}
\begin{alignat}{2}
 a( \bdd{u}, \bdd{v} ) - (q, \dive \bdd{v})  &= L(\bdd{v}) \qquad & &\forall \bdd{v} \in \bdd{H}^1_D(\Omega) , \\
 - (r, \dive \bdd{u}) &= 0 \qquad & & \forall r \in L^2_D(\Omega),
\end{alignat}
\end{subequations}
where
\begin{align}
a(\bdd{u}, \bdd{v}) := ( \bdd{\varepsilon}(\bdd{u}), \bdd{\varepsilon}(\bdd{v}) ) \quad \text{and} \quad
L(\bdd{v}) &:= (\bdd{v}, \bdd{f}) + (\bdd{v}, \bdd{g})_{\Gamma_N} \qquad \forall \bdd{u}, \bdd{v} \in \bdd{H}^1(\Omega)
\end{align}
and $(\cdot,\cdot)_{\omega}$ denotes the $L^2(\omega)$ or $\bdd{L}^2(\omega)$ inner product. More generally,	$|\cdot|_{s,\omega}$ and $\|\cdot\|_{s,\omega}$ denote the $H^s(\omega)$ or $\bdd{H}^s(\omega)$ semi-norm and norm, respectively.  We omit the subscript $\omega$ when $\omega = \Omega$. As a matter of fact, the ensuing discussion will be valid for the more general setting in which the bilinear form $a(\cdot,\cdot)$ satisfies the conditions:
\begin{itemize}
	\item Boundedness: There exists $M > 0$ such that
	\begin{align}
		\label{eq:a bounded}
		|a(\bdd{u}, \bdd{v})| \leq M \|\bdd{u}\|_1 \|\bdd{v}\|_{1} \qquad \forall \bdd{u}, \bdd{v} \in \bdd{H}^1_D(\Omega).
	\end{align}
	
	\item Ellipticity: There exists $\alpha > 0$ such that
	\begin{align}
		\label{eq:a elliptic}
		a(\bdd{u}, \bdd{u}) \geq \alpha \|\bdd{u}\|_1^2 \qquad \forall \bdd{u} \in \bdd{H}^1_D(\Omega). 
	\end{align}

\end{itemize}
In particular, these conditions ensure the well-posedness of \cref{eq:stokes variational form} by standard Babu\v{s}ka-Brezzi theory (see e.g. \cite[Lemma 3.19]{John16}).
Relevant examples of bilinear forms $a(\cdot,\cdot)$ satisfying \cref{eq:a bounded,eq:a elliptic} include
\begin{itemize}
	\item Oseen flow: 
	\begin{align}
		\label{eq:oseen flow}
		a(\bdd{u}, \bdd{v}) = 2\nu( \bdd{\varepsilon}(\bdd{u}), \bdd{\varepsilon}(\bdd{v}) ) + ((\bdd{w}\cdot \nabla) \bdd{u}, \bdd{v}),
	\end{align}
	where $\nu$ is the kinematic viscosity and $\bdd{w}$ is divergence free with $\bdd{w} \cdot \unitvec{n} \geq 0$ on $\Gamma_N$ (see e.g. \cite[\S 1]{John16} for a precise description of the required regularity of $\bdd{w}$).
	
	\item Singular perturbations to Oseen flow: $a(\bdd{u}, \bdd{v}) =  (\bdd{u}, \bdd{v}) + \delta \{ 2\nu ( \bdd{\varepsilon}(\bdd{u}), \bdd{\varepsilon}(\bdd{v}) ) + ((\bdd{w}\cdot \nabla) \bdd{u}, \bdd{v}) \}$, where $\nu$ and $\bdd{w}$ are as above.
\end{itemize}
The Oseen equations arise in numerical methods for the steady Navier-Stokes equations, while the singular perturbation problems arise in time discretizations of unsteady Stokes and Navier-Stokes flow (see e.g. \cite{John16}) in which $\delta \sim \Delta t$, where $\Delta t$ is the timestep.

\subsection{Scott-Vogelius Discretization}
\label{sec:scott vog}
Let $X \subset H^1(\Omega)$ be the set of continuous, piecewise polynomials of degree $p \in \mathbb{N}$ on a triangulation $\mathcal{T}$ of $\Omega$:
\begin{align*}
X := \{ v \in C^0(\bar{\Omega}) : v|_{K} \in \mathcal{P}_{p}(K) \ \forall K \in \mathcal{T} \},
\end{align*}
where $\mathcal{P}_{p}(K)$ denotes the space of polynomials of degree at most $p$. In particular, we assume that the triangulation $\mathcal{T}$ is a partitioning of the domain $\Omega$ into simplices such that the nonempty intersection of any two distinct elements from $\mathcal{T}$ is a single common sub-simplex of both elements with mesh size $h := \max_{K \in \mathcal{T}} h_K$ and $h_K := \mathrm{diam}(K)$. We also assume that element boundaries are located at the intersections of $\bar{\Gamma}_D$ and $\bar{\Gamma}_N$. The space $X_D := X \cap H^1_D(\Omega)$ then consists of functions in $X$ vanishing on the Dirichlet boundary $\Gamma_D$, and we discretize \cref{eq:stokes variational form} using the space $\bdd{X}_D := [X_D]^d$ as follows: Find $(\bdd{u}_{X},q_{X}) \in \bdd{X}_D \times \dive \bdd{X}_D$ such that
\begin{subequations}
\label{eq:stokes variational form fem}
 \begin{alignat}{2}
 \label{eq:stokes variational form fem 1}
 a( \bdd{u}_{X}, \bdd{v} ) - (q_{X}, \dive \bdd{v})  &= L(\bdd{v}) \qquad & &\forall \bdd{v} \in \bdd{X}_D , \\
 \label{eq:stokes variational form fem 2}
 - (r, \dive \bdd{u}_{X}) &= 0 \qquad & & \forall r \in \dive \bdd{X}_D.
 \end{alignat}
\end{subequations}

The pair $\bdd{X}_D \times \dive \bdd{X}_D$ corresponds to the Scott-Vogelius elements \cite{ScottVog84,ScottVog85,Vogelius83divinv} which possess properties that make them an attractive option for mixed high order discretization. Firstly, the velocity space consists of standard continuous finite elements, which are already implemented in most, if not all, high order finite element software packages. Secondly, choosing $r = \dive \bdd{u}_{X}$ in \cref{eq:stokes variational form fem 2} shows that the resulting discrete velocity $\bdd{u}_{X}$ is \textit{pointwise divergence free}, which means that \cref{eq:stokes div free} is satisfied exactly. Moreover, a discrete inf-sup condition holds:
\begin{align}
	\label{eq:discrete inf-sup}
	\beta_X := \inf_{0 \neq q \in \dive \bdd{X}_D} \sup_{ \bdd{v} \in \bdd{X}_D} \frac{(\dive \bdd{v}, q)}{\|\bdd{v}\|_{1} \|q\|},
\end{align}
where, in general, $\beta_X > 0$ depends on $h$ and $p$, but is strictly positive. This is most easily seen by the following argument. The divergence operator $\dive : \bdd{X}_D \to \dive \bdd{X}_D$ is continuous and surjective and $\bdd{X}_D$ is finite dimensional. Thus, the operator $\dive$ admits a bounded right-inverse $R : \dive \bdd{X}_D \to \bdd{X}_D$ with $\dive Rq = q$ for all $q \in \dive \bdd{X}_D$. Choosing $\bdd{v} = Rq$ in the supremum in \cref{eq:discrete inf-sup} gives
\begin{align*}
	\beta_X \geq \inf_{0 \neq q \in \dive \bdd{X}_D}  \frac{(\dive Rq, q)}{\|Rq\|_{1} \|q\|} = \inf_{0 \neq q \in \dive \bdd{X}_D}  \frac{\|q\|}{\|Rq\|_{1}} \geq \| R \|^{-1} > 0,
\end{align*}
where $\|R\|$ denotes the usual operator norm.

\section{Standard Iterated Penalty Method}
\label{sec:vanilla ip}

A classical implementation of the finite element method \cref{eq:stokes variational form fem} would proceed in two steps: (i) selecting a suitable basis for the spaces $\bdd{X}_D$ and $\dive \bdd{X}_D$ and (ii) solving the resulting saddle point system. The standard nature of the velocity space $\bdd{X}_D$ means that a basis may be constructed via the usual techniques. We use the Bernstein basis (see e.g. \cite{Lai07spline}) for the scalar space $X_D$ (other choices are perfectly acceptable). In the case $d=3$, the basis consists of (i) piecewise linear vertex functions, (ii) edge functions, (iii) face functions. and (iv) interior functions, while in the case $d=2$, the face functions play the role of interior degrees of freedom. In particular, there are $d+1$ vertex functions, $\binom{d+1}{2}(p-1)$ edge functions, $\binom{d+1}{3}(p-1)(p-2)/2$ face functions, and $(p-1)(p-2)(p-3)/6$ functions associated with a given element $K \in \mathcal{T}$. A basis for $\bdd{X}_D$ is obtained using functions of the form $\{\phi_j \unitvec{e}_k\}_{k=1}^{d}$, where $\{ \unitvec{e}_k\}_{k=1}^d$ is the standard basis for $\mathbb{R}^d$ and $\{\phi_j\}$ denotes the basis for $X_D$.

In contrast, constructing a basis for the pressure space $\dive \bdd{X}_D$ is far more complicated due, in part, to the large null space of the divergence operator. However, complications also arise from the fact \cite{AinCP21LE,ScottVog84,ScottVog85,Vogelius83divinv} that the dimension of the space $\dive \bdd{X}_D$ is affected by the element topology. For instance, in the case $d=2$ difficulties arise at \textit{singular vertices} \cite{ScottVog85,Vogelius83le}. An element vertex is singular if all element edges meeting at the vertex lie on exactly two straight lines. Thus, in the case of an interior vertex, a singular vertex can only arise when four elements abut the vertex. Together, these features mean that constructing a basis for the pressure space is a much more challenging task compared with constructing a basis for $\bdd{X}_D$. 

The iterated penalty method \cite{Brenner08,FortinGlow83,Glowinski84,MorganScott18} offers an attractive alternative to the classical implementation of \cref{eq:stokes variational form fem} by virtue of the fact that one can circumvent the need to construct an explicit basis for $\dive \bdd{X}_D$ altogether. The iterated penalty method proceeds as follows for a chosen sufficiently large parameter $\lambda > 0$ (see \cref{thm:vanilla ip convergence} below): For $n=0,1,\ldots,$ find $\bdd{u}_{X}^{n} \in \bdd{X}_D$ such that
\begin{subequations}
	\label{eq:vanilla ip}
	\begin{alignat}{2}
	\label{eq:vanilla ip 1}
	a_{\lambda}(\bdd{u}_{X}^{n}, \bdd{v}) &= L(\bdd{v}) + (\dive \bdd{w}_{X}^{n}, \dive \bdd{v}) \qquad & & \forall \bdd{v} \in \bdd{X}_D, \\
	\label{eq:vanilla ip 2}
	\bdd{w}_{X}^{n+1} &= \bdd{w}_{X}^{n} - \lambda \bdd{u}_{X}^{n}, \qquad & &
	\end{alignat}
\end{subequations}
where $\bdd{w}_{X}^{0} := \bdd{0}$ and  
\begin{align}
	\label{eq:alam def}
	a_{\lambda}(\bdd{u}, \bdd{v}) := a( \bdd{u}, \bdd{v} ) + \lambda (\dive \bdd{u}, \dive \bdd{v}) \qquad \forall \bdd{u}, \bdd{v} \in \bdd{H}^1(\Omega).
\end{align}
Note that $\bdd{u}_X^n$ is well-defined by \cref{eq:vanilla ip 1} thanks to the Lax-Milgram lemma since $a(\cdot,\cdot)$, and hence $a_{\lambda}(\cdot,\cdot)$, is elliptic on $\bdd{X}_D \subset \bdd{H}^1_D(\Omega)$. The steps \cref{eq:vanilla ip 1}-\cref{eq:vanilla ip 2} are iterated until a suitable stopping criterion (see \cref{thm:vanilla ip convergence} below for one such criterion) is met, at which point the pressure approximation is taken to be $q_X \simeq q_{X}^{n} := \dive \bdd{w}_{X}^{n}$. The following result concerns the convergence of \cref{eq:vanilla ip}.

\begin{theorem}
	\label{thm:vanilla ip convergence}
	Let $(\bdd{u}_{X}, q_{X}) \in \bdd{X}_D \times \dive \bdd{X}_D$ denote the solution to \cref{eq:stokes variational form fem} and $(\bdd{u}_{X}^{n},  \bdd{w}_{X}^{n})$, $n \in \mathbb{N}$ be given by \cref{eq:vanilla ip}. Then, the following error estimate holds:
	\begin{multline*}
		\max\left\{ \| \bdd{u}_{X} - \bdd{u}_{X}^{n} \|_{1}, \left( \frac{M(M+\alpha)}{\alpha \beta_X^2} + \frac{\sqrt{d} \lambda}{\beta_X} \right)^{-1} \|q_{X} - \dive \bdd{w}_{X}^{n} \| \right\} \\ 
		\leq \frac{M + \alpha}{\alpha \beta_X} \|\dive \bdd{u}_{X}^{n}\|,
	\end{multline*}
	where  $M > 0$ \cref{eq:a bounded}, $\alpha > 0$ \cref{eq:a elliptic}, and $\beta_X > 0$ \cref{eq:discrete inf-sup}.  Moreover,
	\begin{align*}
		\|\dive \bdd{u}_{X}^{n}\| \leq \sqrt{d} \left[ \frac{M(M + \alpha)^2}{\alpha^2 \beta_X^2 \lambda } \right]^n \| \bdd{u}_{X} - \bdd{u}^0 \|_{1}.
	\end{align*}
\end{theorem}

\noindent \Cref{thm:vanilla ip convergence} is proved in \cref{sec:vanilla ip convergence} and shows that, for $\lambda$ sufficiently large, the standard iterated penalty method \cref{eq:vanilla ip} converges at a geometric rate and that the quantity $\|\dive \bdd{u}_{X}^{n}\|$ may be used as the basis for a stopping criterion. 

\subsection{Implementation Cost}

The main cost of using the standard iterated penalty method lies in \cref{eq:vanilla ip 1} which entails solving a square system with $\mathcal{O}(|\mathcal{T}|p^d)$ unknowns at every iteration. The bulk of the degrees of freedom are associated with the interior basis functions which, as remarked earlier, number $\mathcal{O}(p^d)$ per element. In contrast, the number of degrees of freedom associated with element boundaries is $\mathcal{O}(|\mathcal{T}|p^{d-1})$. The question arises: Can system \cref{eq:vanilla ip 1} be reduced to a system of size $\mathcal{O}(|\mathcal{T}|p^{d-1})$ unknowns by an (ideally) one-time elimination, or \textit{static condensation}, of the interior degrees of freedom? 

In order to explore this question, it is convenient to express static condensation in variational form. Given an element $K \in \mathcal{T}$, let
\begin{align}
\label{eq:interior space def}
\bdd{X}_I(K) := \{ \bdd{v} \in \bdd{X}_D : \supp \bdd{v} \subseteq K  \} \quad \text{and} \quad \bdd{X}_I := \bigoplus_{K \in \mathcal{T}} \bdd{X}_I(K).
\end{align}
The orthogonal complement of $\bdd{X}_I$ in $\bdd{X}_D$ with respect to the  $a_{\lambda}(\cdot,\cdot)$ form and its ``adjoint" are given by
\begin{align}
	\label{eq:old boundary left}
\bdd{X}_B &:= \{ \bdd{v} \in \bdd{X}_D : a_{\lambda,K}(\bdd{v}, \bdd{w}) = 0 \ \forall \bdd{w} \in \bdd{X}_I(K), \ \forall K \in \mathcal{T} \} \\
\label{eq:old boundary right}
\bdd{X}_B^{\dagger} &:= \{ \bdd{v} \in \bdd{X}_D : a_{\lambda, K}(\bdd{w}, \bdd{v}) = 0 \ \forall \bdd{w} \in \bdd{X}_I(K), \ \forall K \in \mathcal{T} \},
\end{align}
where $a_{\lambda,K}(\cdot,\cdot)$ is the restriction of $a_{\lambda}(\cdot,\cdot)$ to the element $K$. Static condensation then amounts to seeking the solution to \cref{eq:vanilla ip 1} in the form
\begin{align}
\label{eq:vanilla sc decomp}
\bdd{u}_{X}^{n} = \bdd{u}_B + \sum_{K \in \mathcal{T}} \bdd{u}_K,
\end{align}
in which the contributions are given by
\begin{alignat}{3}
	\label{eq:vanilla ip interior}
	\bdd{u}_K \in \bdd{X}_I(K) &: \ \ & a_{\lambda,K}(\bdd{u}_K, \bdd{v}) &= (\bdd{f}, \bdd{v})_K + (\dive \bdd{w}_{X}^{n}, \dive \bdd{v})_K \quad & &\forall \bdd{v} \in \bdd{X}_I(K), \\
\label{eq:vanilla ip hat}
\bdd{u}_B \in \bdd{X}_B &:  & a_{\lambda}(\bdd{u}_B, \bdd{v}) &= L(\bdd{v}) + (\dive \bdd{w}_{X}^{n}, \dive \bdd{v}) \quad & &\forall \bdd{v} \in \bdd{X}_B^{\dagger}.
\end{alignat} 
The systems \cref{eq:vanilla ip interior} consist of  $\mathcal{O}(p^d)$ interior unknowns on each element that are decoupled and can be solved in parallel using direct methods compared with the global system of $\mathcal{O}(|\mathcal{T}| p^d)$ unknowns corresponding to \cref{eq:vanilla ip 1}. Meanwhile \cref{eq:vanilla ip hat} is equivalent to a global linear system of $\mathcal{O}(|\mathcal{T}|p^{d-1})$ unknowns. \Cref{alg:itpen vanilla} summarizes the standard iterated penalty method in which the solution to \cref{eq:vanilla ip 1} is sought in the form \cref{eq:vanilla sc decomp}. Unfortunately, the computational cost of \cref{alg:itpen vanilla} per iteration remains $\mathcal{O}(|\mathcal{T}| p^{2d})$ operations owing to the need to solve \cref{eq:vanilla ip interior} at every iteration.

\begin{algorithm}[htb]
	\caption{Standard Iterated Penalty Method for \cref{eq:stokes variational form fem}}
	\label{alg:itpen vanilla}
	\begin{algorithmic}[1]
		\Require{ $\bdd{w}_{X}^{0} := \bdd{0}$, $\lambda > 0$}

		\For{$n = 0, 1, \ldots, $}
			\State{Find $\bdd{u}_B \in \bdd{X}_B$ such that
			\begin{align*}
			a_{\lambda}(\bdd{u}_B, \bdd{v}) = L(\bdd{v}) + (\dive \bdd{w}_{X}^{n}, \dive \bdd{v}) \qquad \forall \bdd{v} \in \bdd{X}_B^{\dagger}.
			\end{align*}
			}
			
			\State{For each $K \in \mathcal{T}$, find $\bdd{u}_K \in \bdd{X}_I(K)$ such that
			\begin{align*}
			a_{\lambda,K}(\bdd{u}_K, \bdd{v}) = (\bdd{f}, \bdd{v})_K + (\dive \bdd{w}_{X}^{n}, \dive \bdd{v})_K \qquad \forall \bdd{v} \in \bdd{X}_I(K).
			\end{align*}
			}
			
			\State{$\bdd{u}_{X}^{n} := \bdd{u}_B + \sum_{K \in \mathcal{T}} \bdd{u}_K$}
			
			\If{stopping criteria is met}
			\State{\textbf{break}}
			\EndIf
			
			\State{$\bdd{w}_{X}^{n+1} := \bdd{w}_{X}^{n} - \lambda \bdd{u}_{X}^{n}$}
		\EndFor
		
		\State{\Return{ $\bdd{u}_{X}^{n}$, $q_{X}^{n} := \dive \bdd{w}_{X}^{n}$ }}

	\end{algorithmic}
\end{algorithm}

\section{Reducing the Cost of the Standard Iterated Penalty Method}
\label{sec:modified it}

The foregoing discussion showed that, even with element-wise static condensation, the cost of the standard iterated penalty method remains at $\mathcal{O}(|\mathcal{T}| p^{2d})$ operations per iteration. The main reason why the static condensation failed to reduce the cost per iteration was that lines 4 and 8 in \cref{alg:itpen vanilla} required the values of the interior degrees of freedom at every iteration in order to compute the RHS needed in line 2 for the boundary degrees of freedom. In essence, while static condensation decouples the LHS of the system appearing in \cref{eq:vanilla ip interior,eq:vanilla ip hat}, the problem remains coupled owing to the form of the source terms on the RHS. 

In this section, we show that a judicious modification of the choice the space $\bdd{X}_B$ (and $\bdd{X}_B^{\dagger}$) results in a \textit{full} decoupling of the interior and boundary degrees of freedom. This means that one need only solve for the interior degrees of freedom \textit{once}, as opposed to having to solve for the interiors at every iteration as in \cref{alg:itpen vanilla}. The main idea rests on using properties of the spaces of divergence free interior functions 
\begin{align}
	\label{eq:ni def}
	\bdd{N}_I(K) &:= \{ \bdd{v} \in \bdd{X}_I(K) : \dive \bdd{v} \equiv 0  \}, \ K \in \mathcal{T},  \quad \text{and} \quad \bdd{N}_I := \bigoplus_{K \in \mathcal{T}} \bdd{N}_I(K),
\end{align} 
which will play a key role in analyzing the interior spaces and constructing the appropriate modification to $\bdd{X}_B$.

\subsection{The Interior Spaces}

We start by examining the \textit{Stokes system} associated with the interior degrees of freedom on an element $K \in \mathcal{T}$: Find $(\bdd{u}_K, q_K) \in \bdd{X}_I(K) \times \dive \bdd{X}_I(K)$ such that
\begin{subequations}
	\label{eq:interior stokes discussion}
	\begin{alignat}{2}
		\label{eq:interior stokes discussion 1}
		a_K(\bdd{u}_K, \bdd{v}) - (q_K, \dive \bdd{v})_K &= L_1(\bdd{v}) \qquad & &\forall \bdd{v} \in \bdd{X}_I(K), \\
		\label{eq:interior stokes discussion 2}
		-(r, \dive \bdd{u}_K)_K &= L_2(r)	\qquad & &\forall r \in \dive \bdd{X}_I(K),
	\end{alignat}
\end{subequations}
where $L_1(\cdot)$ and $L_2(\cdot)$ are suitable linear functionals. Problem \cref{eq:interior stokes discussion} may be written in terms of matrices as follows. Any $\bdd{u} \in \bdd{X}_D$ and $q_K \in \dive \bdd{X}_I(K)$, $K \in \mathcal{T}$, may be expressed as
\begin{align*}
	\bdd{u} = \vec{u}_B^T \vec{\Phi}_B + \vec{u}_I^T \vec{\Phi}_I \quad \text{and} \quad q_K = \vec{q}_{K}^T \vec{\psi}_{\iota,K},
\end{align*}
where $\vec{\Phi}_I$ is a basis for the interior velocity functions, $\vec{\Phi}_B$ a basis for the vertex and edge functions, while $\vec{\psi}_{\iota,K}$ is a basis for $\dive \bdd{X}_I(K)$. For $K \in \mathcal{T}$, let $\bdd{E}_K$ be the matrix corresponding to the form $a_K(\cdot,\cdot)$, partitioned as follows:
\begin{align*}
	a_K(\bdd{u}, \bdd{v}) = \begin{bmatrix}
		\vec{v}_{B,K} \\ \vec{v}_{I,K}
	\end{bmatrix}^T \bdd{E}_K \begin{bmatrix}
		\vec{u}_{B,K} \\ \vec{u}_{I,K}
	\end{bmatrix} = \begin{bmatrix}
		\vec{v}_{B,K} \\ \vec{v}_{I,K}
	\end{bmatrix}^T \begin{bmatrix}
		\bdd{E}_{BB} & \bdd{E}_{BI} \\
		\bdd{E}_{IB} & \bdd{E}_{II}
	\end{bmatrix} \begin{bmatrix}
		\vec{u}_{B,K} \\ \vec{u}_{I,K}
	\end{bmatrix} \qquad \forall \bdd{u}, \bdd{v} \in \bdd{X}_D,
\end{align*}
where $\vec{u}_{B,K}$ and $\vec{u}_{I,K}$ are the boundary and interior degrees of freedom of $\bdd{u}$ associated to element $K$. In a similar vein, $\bdd{G}_K$ is the matrix corresponding to $-(\cdot, \dive \cdot)_K$:
\begin{align*}
	-(q_K, \dive \bdd{u})_K = \vec{q}_{K}^T \bdd{G}_K \begin{bmatrix}
		\vec{u}_{B,K} \\ \vec{u}_{I,K}
	\end{bmatrix} = \vec{q}_{K}^T \begin{bmatrix}
		\bdd{G}_{\iota B} &
		\bdd{G}_{\iota I}
	\end{bmatrix} \begin{bmatrix}
		\vec{u}_{B,K} \\ \vec{u}_{I,K}
	\end{bmatrix},
\end{align*}
for all $q_K \in \dive \bdd{X}_I(K)$ and $\bdd{u} \in \bdd{X}_D$. In particular, the LHS of \cref{eq:interior stokes discussion} corresponds to a square matrix 
\begin{align}
	\label{eq:interior stokes matrix}
	\begin{bmatrix}
		\bdd{E}_{II} & \bdd{G}_{\iota I}^T \\
		\bdd{G}_{\iota I} & \bdd{0} 
	\end{bmatrix}.
\end{align}
The first result concerns existence and uniqueness of solutions to \cref{eq:interior stokes discussion}:
\begin{lemma}
	\label{lem:interior inversion}
	The interior Stokes system \cref{eq:interior stokes discussion} is uniquely solvable.
\end{lemma}
\begin{proof}
As shown above, \cref{eq:interior stokes discussion} is equivalent to a square linear system involving the matrix \cref{eq:interior stokes matrix}, and therefore it suffices to show uniqueness. Suppose that $(\bdd{u}_K, q_K) \in \bdd{X}_I(K) \times \dive \bdd{X}_I(K)$ satisfy \cref{eq:interior stokes discussion} with $L_1 = L_2 = 0$. Thanks to \cref{eq:interior stokes discussion 2}, $\bdd{u}_K \in \bdd{N}_I(K)$. Choosing $\bdd{v} = \bdd{u}_K$ in \cref{eq:interior stokes discussion 1} then gives $a(\bdd{u}_K, \bdd{u}_K) = 0$. Since $a(\cdot,\cdot)$ is elliptic on $\bdd{N}_I(K)$ by \cref{eq:a elliptic}, $\bdd{u}_K \equiv 0$. By the definition of $\dive \bdd{X}_I(K)$, there exists $\bdd{w} \in \bdd{X}_I(K)$ such that $\dive \bdd{w} = q_K$, and so \cref{eq:interior stokes discussion 1} gives $0 =(q_K, \dive \bdd{w}) = (q_K, q_K)$. Thus, $q_K \equiv 0$, which completes the proof.
\end{proof}

\subsection{The Boundary Space}
\label{sec:mip subspace decomp}

Previously, in \cref{eq:old boundary left,eq:old boundary right}, the boundary space $\bdd{X}_B$ was chosen to be the orthogonal complement with respect to the form $a_{\lambda}(\cdot,\cdot)$ defined in \cref{eq:alam def}. However, the presence of the term $(\dive \cdot, \dive \cdot)$ in the data in lines 2-3 of \cref{alg:itpen vanilla} was ultimately responsible for the need to recompute the static condensation at each iteration. In order to avoid this dependency on the data, we construct new spaces $\tilde{\bdd{X}}_B$ and $\tilde{\bdd{X}}_B^{\dagger}$ that explicitly decouple the dependency in the data arising from the $(\dive \cdot, \dive \cdot)$ term. In particular, if the space $\tilde{\bdd{X}}_B$ has the property that
\begin{align}
	\label{eq:disc tilde conds}
	\bdd{v} \in \tilde{\bdd{X}}_B \implies (\dive \bdd{v}, q) = 0 \qquad \forall q \in \dive \bdd{X}_I,
\end{align}
then the dependency on the data will be removed. Of course, we still want the space $\tilde{\bdd{X}}_B$ to correspond to degrees of freedom associated with the element boundaries. Therefore, we augment \cref{eq:disc tilde conds} with additional conditions
\begin{align}
	\label{eq:disc tilde conds 2}
	\bdd{v} \in \tilde{\bdd{X}}_B \implies a_{\lambda}(\bdd{v}, \bdd{z}) = 0 \qquad \forall \bdd{z} \in \bdd{N}_I,
\end{align}
where $\bdd{N}_I$ is given by \cref{eq:ni def}. Below, we show that conditions \cref{eq:disc tilde conds,eq:disc tilde conds 2} are independent. Consequently, we arrive at the following choice of the boundary space:
\begin{align}
\label{eq:tildexb def}
\tilde{\bdd{X}}_B &:= \{ \bdd{v} \in \bdd{X}_D : a(\bdd{v}, \bdd{z}) = 0 \ \forall \bdd{z} \in \bdd{N}_I \text{ and } (\dive \bdd{v}, r) = 0 \ \forall r \in \dive \bdd{X}_I \},
\end{align}
where we used the definition of $\bdd{N}_I$ \cref{eq:ni def} to rewrite \cref{eq:disc tilde conds 2} in terms of $a(\cdot,\cdot)$ by dropping the $(\dive \cdot, \dive \cdot)$ term in $a_{\lambda}(\cdot,\cdot)$. Similarly, we define $\tilde{\bdd{X}}_B^{\dagger}$ to be the corresponding ``adjoint" space:
\begin{align*}
\tilde{\bdd{X}}_B^{\dagger} &:= \{ \bdd{v}^{\dagger} \in \bdd{X}_D : a(\bdd{z}, \bdd{v}^{\dagger}) = 0 \ \forall \bdd{z} \in \bdd{N}_I \text{ and } (\dive \bdd{v}^{\dagger}, r) = 0 \ \forall r \in \dive \bdd{X}_I \}.
\end{align*}
The equivalences $\bdd{N}_I = \oplus_{K \in \mathcal{T}} \bdd{N}_I(K)$ and $\bdd{X}_I = \oplus_{K \in \mathcal{T}} \bdd{X}_I(K)$ mean that the conditions appearing in \cref{eq:tildexb def} decouple into independent local conditions for each $K \in \mathcal{T}$:
\begin{align}
	\label{eq:tildexb local}
	a_K(\bdd{v}, \bdd{z}) = 0 \quad \forall \bdd{z} \in \bdd{N}_I(K) \quad \text{and} \quad (\dive \bdd{v}, r)_K = 0 \quad \forall r \in \dive \bdd{X}_I(K).
\end{align}
Moreover, conditions \cref{eq:tildexb local} are linearly independent of one another. This can most easily be seen from the matrix form of \cref{eq:tildexb local} which reads
\begin{alignat*}{2}
	\vec{z}_{I,K}^T \bdd{E}_{II} \vec{v}_{I,K} &= -\vec{z}_{I,K}^T \bdd{E}_{IB} \vec{v}_{B,K} \qquad & & \forall \bdd{z} \in \bdd{N}_I(K) \\
	\vec{r}_K^T \bdd{G}_{\iota I} \vec{v}_{I,K} &= - \vec{r}_K^T \bdd{G}_{\iota B} \vec{v}_{B,K} \qquad & &\forall r \in \dive \bdd{X}_I(K).
\end{alignat*}
Note that for any $\bdd{z} \in \bdd{N}_I(K)$ and  $s \in \dive \bdd{X}_I(K)$, $\vec{z}_{I,K}^T \bdd{G}_{\iota I}^T \vec{s}_K = 0$ by definition, and so we equivalently have
\begin{align}
	\label{eq:tildexb local matrix}
	\begin{bmatrix}
		\vec{z}_{I,K} \\ \vec{r}_K
	\end{bmatrix}^T \begin{bmatrix}
		\bdd{E}_{II} & \bdd{G}_{\iota I}^T \\
		\bdd{G}_{\iota I} & \bdd{0} 
	\end{bmatrix} \begin{bmatrix}
	\vec{v}_{I,K} \\ *
\end{bmatrix} = -\begin{bmatrix}
\vec{z}_{I,K} \\ \vec{r}_K
\end{bmatrix}^T \begin{bmatrix}
\bdd{E}_{IB} \\ \bdd{G}_{\iota B}
\end{bmatrix} \vec{v}_{B,K}
\end{align}
for all $(\bdd{z}, r) \in \bdd{N}_I(K) \times \dive \bdd{X}_I(K)$. Here, $*$ denotes an unimportant (but appropriately sized) vector. By \cref{lem:interior inversion}, the matrix \cref{eq:interior stokes matrix} appearing on the LHS above is invertible, which means that the conditions in \cref{eq:tildexb local} are indeed linearly independent. Moreover, we have the following inclusion:
\begin{align}
	\label{eq:matrix relation incl tildexb}
	\{ \bdd{v} \in \bdd{X}_D : \vec{v}_{I,K} = \bdd{S}_K \vec{v}_{B,K} \ \forall K \in \mathcal{T} \} \subseteq \tilde{\bdd{X}}_B, 
\end{align}
where
\begin{align}
	\label{eq:zk def}
	\bdd{S}_K := - \begin{bmatrix}
		\bdd{I} & \bdd{0}
	\end{bmatrix} \begin{bmatrix}
		\bdd{E}_{II} & \bdd{G}_{\iota I}^T \\
		\bdd{G}_{\iota I} & \bdd{0} 
	\end{bmatrix}^{-1} \begin{bmatrix}
		\bdd{E}_{IB} \\ \bdd{G}_{\iota B}
	\end{bmatrix}.
\end{align}
As we later show (in \cref{lem:tildexb matrix characterization}), the reverse inclusion also holds, meaning that \cref{eq:matrix relation incl tildexb} holds as an equality. 

A similar characterization is obtained for $\tilde{\bdd{X}}_B^{\dagger}$ by first expressing conditions \cref{eq:tildexb local matrix} as
\begin{align}
	\label{eq:tildexb local matrix big}
	\begin{bmatrix}
		\begin{array}{c}
			\vec{0} \\ \vec{0} \\ \hline \vec{z}_{I,K} \\ \vec{r}_K
		\end{array}
	\end{bmatrix}^T
	\begin{bmatrix}
		\begin{array}{cc|cc}
			* & * & \bdd{E}_{BI} & \bdd{G}_{\iota B}^T \\
			* & * &  * & \bdd{0}\\
			\hline 
			\bdd{E}_{IB} & * &  \bdd{E}_{II} & \bdd{G}_{\iota I}^T \\
			\bdd{G}_{\iota B} & \bdd{0} & \bdd{G}_{\iota I} & \bdd{0}
		\end{array}
	\end{bmatrix}
	\begin{bmatrix}
		\begin{array}{c}
			\vec{v}_{B,K} \\ \vec{0} \\ \hline \vec{v}_{I,K} \\ *
		\end{array}
	\end{bmatrix} = 0
\end{align}
for all $(\bdd{z}, r) \in \bdd{N}_I(K) \times \dive \bdd{X}_I(K)$, where we again use $*$ to denote unimportant (but again appropriately sized) vectors or matrices. Using similar arguments, we may show that the conditions for the adjoint space $\tilde{\bdd{X}}_B^{\dagger}$ are the transpose of the conditions in \cref{eq:tildexb local matrix big}, which leads to the following relation:
\begin{align*}
	\{ \bdd{v} \in \bdd{X}_D : \vec{v}_{I,K} = \bdd{T}_K \vec{v}_{B,K} \ \forall K \in \mathcal{T} \} \subseteq \tilde{\bdd{X}}_B^{\dagger}, 
\end{align*}
where
\begin{align}
	\label{eq:sk def}
	\bdd{T}_K := - \begin{bmatrix}
		\bdd{I} & \bdd{0}
	\end{bmatrix} \begin{bmatrix}
		\bdd{E}_{II}^T & \bdd{G}_{\iota I}^T \\
		\bdd{G}_{\iota I} & \bdd{0} 
	\end{bmatrix}^{-1} \begin{bmatrix}
		\bdd{E}_{BI}^T \\ \bdd{G}_{\iota B}
	\end{bmatrix}.
\end{align}
Note that $\bdd{T}_K$ is well-defined since the matrix appearing in \cref{eq:sk def} is the transpose of the (invertible) matrix \cref{eq:interior stokes matrix}. In summary, we have
\begin{lemma}
	\label{lem:tildexb matrix characterization}
	$\tilde{\bdd{X}}_B = \{ \bdd{v} \in \bdd{X}_D : \vec{v}_{I,K} = \bdd{S}_K \vec{v}_{B,K} \ \forall K \in \mathcal{T} \}$ and $\tilde{\bdd{X}}_B^{\dagger} = \{ \bdd{v} \in \bdd{X}_D : \vec{v}_{I,K} = \bdd{T}_K \vec{v}_{B,K} \ \forall K \in \mathcal{T} \}$.
\end{lemma}
\begin{proof}
	Let $\bdd{v} \in \tilde{\bdd{X}}_B$ and define $\bdd{w} \in \bdd{X}_D$ by the rule $\bdd{w} := \vec{\Phi}_B^T \vec{v}_B + \vec{\Phi}_I^T \vec{w}_I$, where $\vec{w}_{I,K} := \bdd{S}_K \vec{v}_{B,K}$ for all $K \in \mathcal{T}$. By \cref{eq:matrix relation incl tildexb}, $\bdd{w} \in \tilde{\bdd{X}}_B$. The function $\bdd{X}_D \ni \bdd{e} := \bdd{v} - \bdd{w} = \vec{\Phi}_I^T (\vec{v}_I - \vec{w}_I)$ then satisfies $\bdd{e} \in \tilde{\bdd{X}}_B$ by linearity and $\bdd{e} \in \bdd{X}_I$ since the boundary degrees of freedom of $\bdd{e}$ are identically zero. By the second condition in the definition of $\tilde{\bdd{X}}_B$ \cref{eq:tildexb def}, $\|\dive \bdd{e}\|^2 = 0$ and so $\bdd{e} \in \bdd{N}_I$. The first condition in the definition of $\tilde{\bdd{X}}_B$ gives $a(\bdd{e}, \bdd{e}) = 0$, and so $\bdd{e} \equiv \bdd{0}$ thanks to \cref{eq:a elliptic}. Consequently, $\bdd{v} = \bdd{w}$ and $\tilde{\bdd{X}}_B = \{ \bdd{v} \in \bdd{X}_D : \vec{v}_{I,K} = \bdd{S}_K \vec{v}_{B,K} \ \forall K \in \mathcal{T} \}$. Similar arguments show that $\tilde{\bdd{X}}_B^{\dagger} = \{ \bdd{v} \in \bdd{X}_D : \vec{v}_{I,K} = \bdd{T}_K \vec{v}_{B,K} \ \forall K \in \mathcal{T} \}$.
\end{proof}
\Cref{lem:tildexb matrix characterization} confirms the expectation that the spaces $\tilde{\bdd{X}}_B$ and $\tilde{\bdd{X}}_B^{\dagger}$ are associated with element boundaries: i.e. the interior degrees of freedom of a function in $\tilde{\bdd{X}}_B$ or $\tilde{\bdd{X}}_B^{\dagger}$ are uniquely determined by its boundary degrees of freedom, which, in turn, means that $\bdd{X}_D = \bdd{X}_I \oplus \tilde{\bdd{X}}_B = \bdd{X}_I \oplus \tilde{\bdd{X}}_B^{\dagger}$.

We record this result, along with some useful properties of the spaces $\tilde{\bdd{X}}_B$ and $\dive \tilde{\bdd{X}}_B$ which we shall need shortly:
\begin{theorem}
	\label{thm:xd decomp and inf-sup}
	There holds
	\begin{align}
	\label{eq:xd decomp}
	\bdd{X}_D = \bdd{X}_I \oplus \tilde{\bdd{X}}_B \quad \text{and} \quad \dive \bdd{X}_D = \dive \bdd{X}_I \oplus \dive \tilde{\bdd{X}}_B.
	\end{align}
	Moreover, the pair $\tilde{\bdd{X}}_B \times \dive \tilde{\bdd{X}}_B$ satisfies an inf-sup condition:
	\begin{align}
	\label{eq:inf-sup boundary}
	\frac{\alpha \beta_X}{M + \alpha} \|\tilde{r}\| &\leq \sup_{\bdd{0} \neq \tilde{\bdd{v}} \in \tilde{\bdd{X}}_B} \frac{(\dive \tilde{\bdd{v}}, \tilde{r})}{\|\tilde{\bdd{v}}\|_1} \qquad \forall \tilde{r} \in \dive \tilde{\bdd{X}}_B, 
	\end{align}
	where $M > 0$ \cref{eq:a bounded}, $\alpha > 0$ \cref{eq:a elliptic}, and $\beta_X > 0$ \cref{eq:discrete inf-sup}. \Cref{eq:xd decomp,eq:inf-sup boundary} also hold with $\tilde{\bdd{X}}_B$ replaced by $\tilde{\bdd{X}}_B^{\dagger}$.
\end{theorem}
\begin{proof}
	As mentioned above, the decomposition $\bdd{X}_D = \bdd{X}_I \oplus \tilde{\bdd{X}}_B = \bdd{X}_I \oplus \tilde{\bdd{X}}_B^{\dagger}$ follows from \cref{lem:tildexb matrix characterization}. Consequently, $\dive \bdd{X}_D = \dive \bdd{X}_I \oplus \dive \tilde{\bdd{X}}_B$.

	Let $\tilde{r} \in \dive \tilde{\bdd{X}}_B$ be given. By \cref{eq:discrete inf-sup}, there exists $\bdd{w} \in \bdd{X}_D$ such that $\dive \bdd{w} = \tilde{r}$ and $\|\bdd{w}\|_1 \leq \beta_X^{-1} \|\tilde{r}\|$. Thanks to \cref{eq:a elliptic,eq:a bounded}, there exists $\bdd{z}_I \in \bdd{N}_I$ such that $a(\bdd{z}, \bdd{n}) = a(\bdd{w}, \bdd{n})$ for all $\bdd{n} \in \bdd{N}_I$ satisfying $\|\bdd{z}\|_1 \leq M \alpha^{-1} \|\bdd{w}\|_1$ by the Lax-Milgram Lemma. The function $\tilde{\bdd{v}} := \bdd{w} - \bdd{z}$ then satisfies $\dive \tilde{\bdd{v}} = \tilde{r}$, $\tilde{\bdd{v}} \in \tilde{\bdd{X}}_B$, and $\|\bdd{v}\|_1  \leq (M+\alpha)/(\alpha\beta_X)\|\tilde{r}\|$. Given $\tilde{q} \in \dive \tilde{\bdd{X}}_B^{\dagger}$, a function $\bdd{v}^{\dagger} \in \tilde{\bdd{X}}_B^{\dagger}$ satisfying $\dive \bdd{v}^{\dagger} = \tilde{q}$ and $\|\bdd{v}^{\dagger}\|_1 \leq (M+\alpha)/(\alpha\beta_X)\|\tilde{q}\|$ may be constructed analogously.
\end{proof}

\subsection{The Statically Condensed Iterated Penalty Method}

Similarly to \cref{eq:vanilla ip hat,eq:vanilla ip interior}, the decomposition \cref{eq:xd decomp} decouples the solution to \cref{eq:stokes variational form fem} into its boundary and interior components. However, this time there is a crucial difference in that the interior components $\{ \bdd{u}_K \}$ defined in \cref{eq:stokes interior} do not appear in the data for the system \cref{eq:stokes tilde spaces} determining the boundary component $\tilde{\bdd{u}}$:
\begin{lemma}
	\label{lem:stokes sol decomp}
	The solution to \cref{eq:stokes variational form fem} may be written in the form 
	\begin{align*}
	\bdd{u}_{X} := \tilde{\bdd{u}} + \sum_{K \in \mathcal{T}}\bdd{u}_K \quad \text{and} \quad q_{X} := \tilde{q} + \sum_{K \in \mathcal{T}} q_{K},
	\end{align*}
	where:
	
	1. $(\tilde{\bdd{u}}, \tilde{q}) \in \tilde{\bdd{X}}_B \times \dive \tilde{\bdd{X}}_B$ satisfy
	\begin{subequations}
		\label{eq:stokes tilde spaces}
		\begin{alignat}{2}
			\label{eq:stokes tilde spaces 1}
			a( \tilde{\bdd{u}}, \bdd{v} ) - (\tilde{q}, \dive \bdd{v})  &= L(\bdd{v}) \qquad & &\forall \bdd{v} \in \tilde{\bdd{X}}_B^{\dagger}, \\
			\label{eq:stokes tilde spaces 2}
			- (r, \dive \tilde{\bdd{u}}) &= 0 \qquad & & \forall r \in \dive \tilde{\bdd{X}}_B^{\dagger}.
		\end{alignat}
	\end{subequations}
	
	\noindent 2. For each $K \in \mathcal{T}$, $(\bdd{u}_{K}, q_{K}) \in \bdd{X}_I(K) \times \dive \bdd{X}_I(K)$ satisfy
	\begin{subequations}
		\label{eq:stokes interior}
		\begin{alignat}{2}
			a_K( \bdd{u}_{K}, \bdd{v} ) - (q_{K}, \dive \bdd{v})_K  &= (\bdd{f}, \bdd{v})_K - a_K( \tilde{\bdd{u}}, \bdd{v}) \qquad & &\forall \bdd{v} \in \bdd{X}_I(K), \\
			- (r, \dive \bdd{u}_{K})_K &= 0 \qquad & & \forall r \in \dive \bdd{X}_I(K).
		\end{alignat}
	\end{subequations}
	Moreover, the systems \cref{eq:stokes tilde spaces} and \cref{eq:stokes interior} are uniquely solvable.
\end{lemma}
\Cref{lem:stokes sol decomp}, whose proof is given in \cref{sec:proof decomp}, shows the finite element solution $(\bdd{u}_{X},q_{X})$ to \cref{eq:stokes variational form fem} may be computed by first solving a global Stokes system posed on the boundary spaces $\tilde{\bdd{X}}_B \times \dive \tilde{\bdd{X}}_B$ \cref{eq:stokes tilde spaces} and then solving decoupled local Stokes systems posed on local interior spaces $\bdd{X}_I(K) \times \dive \bdd{X}_I(K)$ \cref{eq:stokes interior}. Crucially, the data in \cref{eq:stokes tilde spaces} which determines the boundary unknowns $\tilde{\bdd{u}}$ is independent of the interior problem \cref{eq:stokes interior}. In other words, the system \cref{eq:stokes tilde spaces} which determines the boundary degrees of freedom can now be solved independently of \cref{eq:stokes interior}. By way of contrast, this was not the case previously when static condensation was based on $\bdd{X}_B$.

The systems \cref{eq:stokes tilde spaces,eq:stokes interior} now take the form of \textit{Stokes problems} posed over the spaces $\tilde{\bdd{X}}_B \times \dive \tilde{\bdd{X}}_B$ and $\bdd{X}_I(K) \times \dive \bdd{X}_I(K)$. In particular, the construction of a basis for $\dive \tilde{\bdd{X}}_B$ inherits all of the difficulties already mentioned when discussing $\dive \bdd{X}_D$, which led us to consider using the standard iterated penalty method in the first place. However, by the same token, we may solve the global system \cref{eq:stokes tilde spaces} using the standard iterated penalty method with the crucial difference that there is no need to perform static condensation during the iteration. Instead, the interior degrees of freedom are computed once after the boundary component $\tilde{\bdd{u}}$ is in hand by solving \cref{eq:stokes interior}. 

The problem of solving the interior problems \cref{eq:stokes interior} posed over the spaces $\bdd{X}_I(K) \times \dive \bdd{X}_I(K)$ remains. One could, of course, solve the problem using the iterated penalty method, but this would lead to having to iterate over problems of size $\mathcal{O}(p^d)$, which is precisely what we are seeking to avoid. Fortunately, as shown in \cite[Lemma 2.5]{Vogelius83divinv} in the case $d=2$, the local interior pressure space can be characterized explicitly as follows: 
\begin{align}
	\label{eq:divxi character}
\dive \bdd{X}_I(K) = \{ r \in \mathcal{P}_{p-1}(K) \cap L^2_0(K) : r(\bdd{a}) = 0 \text{ for all vertices $\bdd{a}$ of $K$}  \}.
\end{align}
Similarly, in the case $d=3$, one has \cite[Theorem 4.2]{Neilan15}:
\begin{align*}
	\dive \bdd{X}_I(K) = \{ q \in \mathcal{P}_{p-1}(K) \cap L^2_0(K): q \text{ vanishes along element edges} \}.
\end{align*}
\noindent These characterizations mean that a basis for $\dive \bdd{X}_I(K)$ may be constructed via standard methods; see e.g.  \cref{sec:impl notes} for a basis using Bernstein polynomials in the case $d=2$. Consequently, one can assemble and  invert the local systems \cref{eq:stokes interior} directly proceeding element-by-element. 

The overall scheme, dubbed the Statically Condensed Iterated Penalty (SCIP) method, is summarized in \cref{alg:itpen modified}. Crucially, the solve for the interior degrees of freedom now happens \textit{outside} the \texttt{for} loop in \cref{alg:itpen modified}, which means that \textit{each iteration of the standard iterated penalty method applied to \cref{eq:stokes tilde spaces} only entails inverting a linear system of $\mathcal{O}(|\mathcal{T}| p^{d-1})$ unknowns, compared to inverting a system with $\mathcal{O}(|\mathcal{T|} p^d)$ unknowns for the standard iterated penalty method applied to \cref{eq:stokes variational form fem}}. Moreover, inf-sup condition for $\tilde{\bdd{X}}_B \times \dive \tilde{\bdd{X}}_B$ \cref{eq:inf-sup boundary} gives the following analogue of \cref{thm:vanilla ip convergence}:
\begin{theorem}
	\label{thm:boundary ip convergence}
	Let $(\tilde{\bdd{u}}, \tilde{q}) \in \tilde{\bdd{X}}_B \times \dive \tilde{\bdd{X}}_B$ be the solution to \cref{eq:stokes tilde spaces} and $(\tilde{\bdd{u}}^n, \tilde{\bdd{w}}^n)$, $n \in \mathbb{N}$ be given by \cref{alg:itpen modified}. Then, there holds
	\begin{multline}
	\label{eq:scip convergence}
	\max\left\{ \|\tilde{\bdd{u}} - \tilde{\bdd{u}}^n \|_{1}, \left( \frac{M(M+\alpha)^3}{\alpha^3 \beta_X^2} + \frac{\sqrt{d} \lambda(M+\alpha)}{\alpha \beta_X} \right)^{-1} \|\tilde{q} - \dive \tilde{\bdd{w}}^n \| \right\} \\
	\leq \frac{(M + \alpha)^2}{\alpha^2 \beta_X} \|\dive \tilde{\bdd{u}}^n\|,
	\end{multline}
	where $M > 0$ \cref{eq:a bounded}, $\alpha > 0$ \cref{eq:a elliptic}, and $\beta_X > 0$ \cref{eq:discrete inf-sup}. Moreover,
	\begin{align}
	\label{eq:scip div convergence}
	\|\dive \tilde{\bdd{u}}^n\| \leq \sqrt{d} \left[ \frac{M(M + \alpha)^4}{\alpha^4 \beta_X^2 \lambda } \right]^n \| \tilde{\bdd{u}} - \tilde{\bdd{u}}^0 \|_{1}.
	\end{align}
\end{theorem}
The presence of the adjoint spaces $\tilde{\bdd{X}}_B^{\dagger}$ and $\dive \tilde{\bdd{X}}_B^{\dagger}$ in \cref{eq:stokes tilde spaces} means that \cref{thm:boundary ip convergence} is not an immediate consequence of results for the standard iterated penalty method e.g. \cite[Theorem 13.1.19 \& Theorem 13.2.2]{Brenner08}, and a short proof is therefore given in \cref{sec:scip convergence}. In order to obtain a geometric rate of convergence, the parameter $\lambda$ must be chosen so that $\lambda \geq M(M+\alpha)^2(\alpha \beta_X)^{-2}$ for the standard iterated penalty method \cref{alg:itpen vanilla}, whereas $\lambda$ must be chosen slightly larger with $\lambda \geq M(M+\alpha)^4 (\alpha^2 \beta_X)^{-2}$ for \cref{alg:itpen modified}. 

\begin{algorithm}[htb]
	\caption{Statically Condensed Iterated Penalty Method (SCIP) for \cref{eq:stokes variational form fem}}
	\label{alg:itpen modified}
	\begin{algorithmic}[1]
		\Require{ $\tilde{\bdd{w}}^0 := \bdd{0}$, $\lambda > 0$}
		
		\For{$n = 0, 1, \ldots, $}
		\State{Find $\tilde{\bdd{u}}^n \in \tilde{\bdd{X}}_B$ such that
		\begin{align}
		\label{eq:boundary ip 1}
		a_{\lambda}(\tilde{\bdd{u}}^n, \bdd{v}) = L(\bdd{v}) + (\dive \tilde{\bdd{w}}^n, \dive \bdd{v}) \qquad \forall \bdd{v} \in \tilde{\bdd{X}}_B^{\dagger}.
		\end{align}
	}
		
		\If{stopping criteria is met} 
		\State{\textbf{break}}
		\EndIf
		
		\State{$\tilde{\bdd{w}}^{n+1} := \tilde{\bdd{w}}^n - \lambda \tilde{\bdd{u}}^n$}
		\EndFor
		
		\State{For each $K \in \mathcal{T}$, find $(\bdd{u}_{K}, q_{K}) \in \bdd{X}_I(K) \times \dive \bdd{X}_I(K)$ such that
			\begin{subequations}
				\begin{alignat*}{2}
				a_K( \bdd{u}_{K}, \bdd{v} ) - (q_{K}, \dive \bdd{v})_K  &= (\bdd{f}, \bdd{v})_K - a_K( \tilde{\bdd{u}}^n, \bdd{v} ) \qquad & &\forall \bdd{v} \in \bdd{X}_I(K), \\
				- (r, \dive \bdd{u}_{K})_K &= 0 \qquad & & \forall r \in \dive \bdd{X}_I(K).
				\end{alignat*}
		\end{subequations}}
		
		\State{\Return{ $\bdd{u}_{X}^{n} := \tilde{\bdd{u}}^n + \sum_{K \in \mathcal{T}} \bdd{u}_K$, $q_{X}^{n} := \dive \tilde{\bdd{w}}^n + \sum_{K \in \mathcal{T}} q_K$ }}

	\end{algorithmic}
\end{algorithm}

\subsection{Matrix Form of SCIP}
\label{sec:implementation details}
In order to facilitate the implementation of the SCIP method, we now derive the matrix form of \cref{alg:itpen modified}.

\textbf{Stiffness Matrices and Load Vectors.} We first consider the bilinear forms and load vectors in line 2 of \cref{alg:itpen modified}. Let $\bdd{E}_K$ and $\bdd{G}_K$ be defined and partitioned as in \cref{sec:mip subspace decomp}, and let $\bdd{C}_K$ correspond to the form $(\dive \cdot, \dive \cdot)_K$, partitioned analogously, where we use the superscript ``$(K)$" to explicitly indicate the dependence of matrix and vector sub-blocks on the element $K$:
\begin{align*}
	(\dive \bdd{u}, \dive \bdd{v})_K =  \begin{bmatrix}
		\vec{v}_{B,K} \\ \vec{v}_{I,K}
	\end{bmatrix}^T \begin{bmatrix}
		\bdd{C}_{BB}^{(K)} & \bdd{C}_{BI}^{(K)} \\
		\bdd{C}_{IB}^{(K)} & \bdd{C}_{II}^{(K)}
	\end{bmatrix} \begin{bmatrix}
		\vec{u}_{B,K} \\ \vec{u}_{I,K}
	\end{bmatrix} \qquad \forall \bdd{u}, \bdd{v} \in \bdd{X}_D.
\end{align*}
Likewise, let $\vec{L}_K$ denote the element load vector satisfying corresponding to the data $\bdd{f}$ and $\bdd{g}$:
\begin{align*}
	 (\bdd{f}, \bdd{v})_K + (\bdd{g}, \bdd{v})_{\Gamma_N \cap \partial K} = L_K(\bdd{v}) = \begin{bmatrix}
	 	\vec{v}_{B,K} \\
	 	\vec{v}_{I,K}
	 \end{bmatrix}^T \vec{L}_K = \begin{bmatrix}
		\vec{v}_{B,K} \\
		\vec{v}_{I,K}
	\end{bmatrix}^T \begin{bmatrix}
		\vec{L}_B^{(K)} \\
		\vec{L}_I^{(K)}
	\end{bmatrix} \qquad \forall \bdd{v} \in \bdd{X}_D.
\end{align*} 
With the element matrices and load vectors in hand, we define
\begin{align*}
	\tilde{\bdd{E}}_K &:= \bdd{E}_{BB}^{(K)} + \bdd{E}_{BI}^{(K)} \bdd{S}_K + \bdd{T}_K^T \bdd{E}_{IB}^{(K)} + \bdd{T}_K^T \bdd{E}_{II}^{(K)} \bdd{S}_K, \\
	\tilde{\bdd{C}}_K &:= \bdd{C}_{BB}^{(K)} + \bdd{C}_{BI}^{(K)} \bdd{S}_K + \bdd{T}_K^T \bdd{C}_{IB}^{(K)} + \bdd{T}_K^T \bdd{C}_{II}^{(K)} \bdd{S}_K, \\
	\tilde{\bdd{A}}_K &:= \tilde{\bdd{E}}_K + \lambda 	\tilde{\bdd{C}}_K, \\
	\vec{\tilde{L}}_K &:= \vec{L}_B^{(K)} + \bdd{T}_K^T \vec{L}_I^{(K)},
\end{align*}
where $\bdd{S}_K$ and $\bdd{T}_K$ are defined in \cref{eq:sk def,eq:zk def}
Thanks to \cref{lem:tildexb stokes ext}, we have the following relations for all $\tilde{\bdd{u}}\in \tilde{\bdd{X}}_B$ and $\tilde{\bdd{v}} \in \ \tilde{\bdd{X}}_B^{\dagger}$:
\begin{align*}
a_{\lambda,K}(\tilde{\bdd{u}}, \tilde{\bdd{v}}) &= \vec{v}_{B,K}^T \tilde{\bdd{A}}_K \vec{u}_{B,K}, \ \ 
L_K(\tilde{\bdd{v}}) = \vec{v}_{B,K}^T \vec{\tilde{L}}_K,  \ \ (\dive \tilde{\bdd{u}}, \dive \tilde{\bdd{v}})_K = \vec{v}_{B,K}^T \tilde{\bdd{C}}_K \vec{u}_{B,K}.
\end{align*}
The local matrices $\tilde{\bdd{A}}_K$ and $\tilde{\bdd{C}}_K$ and the load vector $\tilde{L}_K$ are sub-assembled in the usual way to obtain the global matrices $\tilde{\bdd{A}}$ and $\tilde{\bdd{C}}$ and the global load vector $\vec{\tilde{L}}$. Given $\tilde{\bdd{w}}^n \in \tilde{\bdd{X}}_B$, line 2 of \cref{alg:itpen modified} corresponds to line 2 of \cref{alg:itpen matrix}. We again emphasize that line 2 of \cref{alg:itpen matrix} consists of inverting a system of $\mathcal{O}(|\mathcal{T}| p^{d-1})$ unknowns at each iteration, while lines 2-3 of the standard iterated penalty method consists of inverting a system of $\mathcal{O}(|\mathcal{T}| p^d)$ unknowns at each iteration.

\textbf{Local Stokes Systems.} For each $K \in \mathcal{T}$, the element-wise system in line 8 of \cref{alg:itpen modified} corresponds to line 8 of \cref{alg:itpen matrix}. The associated systems can be solved in parallel using a direct solver. In particular, observe that the interior degrees of freedom are not updated during each iteration.
 
\textbf{Solution Representation.} The final step in line 9 of \cref{alg:itpen modified} entails expressing the solution $\bdd{u}_{X}^{n}$ and $q_{X}^{n}$ with respect to some bases. For simplicity, we give the degrees of freedom on each element $K \in \mathcal{T}$. For the velocity $\bdd{u}_{X}^{n}$, it is convenient to use the original basis for $\bdd{X}_D$ restricted to $K$. By \cref{lem:tildexb matrix characterization}, we have
\begin{align*}
\bdd{u}_{X}^{n}|_{K} &= \vec{\Phi}_{B,K}^T \vec{u}_{B,K}^n + \vec{\Phi}_{I,K}^T \left( \vec{u}_{K} + \bdd{S}_K \vec{u}_{B,K}^n \right), \\ \tilde{\bdd{w}}^n|_{K} &= \vec{\Phi}_{B,K}^T \vec{w}_{B,K}^n + \vec{\Phi}_{I,K}^T \bdd{S}_K \vec{w}_{B,K}^n.
\end{align*}
For the pressure $q$, we take $\{ \psi_{B,K} \} \subset \mathcal{P}_{p-1}(K)$ to be any linearly independent set of $\dim \mathcal{P}_{p-1}(K) - \dim \dive \bdd{X}_I(K)$ functions such that $\{ \psi_{\iota,K} \} \cup \{ \psi_{B,K} \}$ is a basis for $\mathcal{P}_{p-1}(K)$. Then, there exists a matrix $\bdd{H}_K$ such that 
\begin{align*}
\dive \bdd{v}|_{K} = \begin{bmatrix}
\psi_{B,K} \\ \psi_{\iota,K}
\end{bmatrix}^T \bdd{H}_K \vec{v}_K \qquad \forall \bdd{v} \in \bdd{X}_D,
\end{align*}
and so line 9 of \cref{alg:itpen matrix} corresponds to line 9 of \cref{alg:itpen modified}.

\begin{algorithm}[htb]
	\caption{Matrix Form of the SCIP Method}
	\label{alg:itpen matrix}
	\begin{algorithmic}[1]
		\Require{ $\vec{w}_{B}^0 = \vec{0}$, $\lambda > 0$}
		
		\For{$n = 0, 1, \ldots, $}
		\State{Solve $\tilde{\bdd{A}} \vec{u}_B^n = \vec{\tilde{L}} + \tilde{\bdd{C}} \vec{w}_B^{n}$}
		
		\If{stopping criteria is met} 
		\State{\textbf{break}}
		\EndIf
		
		\State{$\vec{w}_B^{n+1} := \vec{w}_B^n - \lambda \vec{u}_B^n$}
		\EndFor
		
		\State{For each $K \in \mathcal{T}$, solve $\begin{bmatrix}
				\bdd{E}_{II}^{(K)} & (\bdd{G}_{\iota I}^{(K)})^T \\
				\bdd{G}_{\iota I}^{(K)} & \bdd{0}
			\end{bmatrix} \begin{bmatrix}
				\vec{u}_{K} \\ \vec{q}_K
			\end{bmatrix} = \begin{bmatrix}
				\vec{L}_I^{(K)} - \bdd{E}_{IB}^{(K)} \vec{u}_{B,K}^n \\
				\vec{0}
			\end{bmatrix}$.}
		
		\State{\Return{ $\begin{bmatrix}
					\vec{u}_{B,K}^n \\ \vec{u}_K + \bdd{S}_K \vec{u}_{B,K}^n
				\end{bmatrix}$ and $ \bdd{H}_K \begin{bmatrix}
			\vec{w}_{B,K}^n \\
			\bdd{S}_K \vec{w}_{B,K}^n
		\end{bmatrix} + \begin{bmatrix}
		\vec{0} \\ \vec{q}_K
	\end{bmatrix}$, $K \in \mathcal{T}$. }}

	\end{algorithmic}
\end{algorithm}

\subsection{Generalization to Other Finite Elements}

The foregoing discussion readily extends to any conforming finite element space $\bdd{X}_D$ such that $\bdd{X}_I$ defined by \cref{eq:interior space def} is nonempty. In particular, all of the results of the current section, \cref{sec:stokes extension}, and \cref{sec:scip convergence} are valid with identical proofs, where the inf-sup constant $\beta_X$ \cref{eq:discrete inf-sup} now corresponds to the pair $\bdd{X}_D \times \dive \bdd{X}_D$, which may depend on $h$ and $p$. Of course, implementing the SCIP method requires an explicit characterization of the space $\dive \bdd{X}_I(K)$, which may not be known or available in all cases. 

\section{Numerical Examples}
\label{sec:numerics}

\tikzset{boximg/.style={remember picture,red,thick,draw,inner sep=0pt,outer sep=0pt}}

We now present two numerical examples highlighting the convergence properties of the SCIP algorithm applied to the 2D Scott-Vogelius elements with $p \geq 4$. As shown in \cref{thm:optimal approx}, these elements are uniformly inf-sup stable in the mesh size $h$ and the polynomial degree $p$ and possess optimal approximation properties on a wide class of meshes. Consequently, estimate \cref{eq:scip convergence} ensures that the convergence of the SCIP method (in exact arithmetic) will not degrade as the mesh is refined or as the polynomial degree is increased. In particular, this choice of element allows us to examine the performance of SCIP independently of problems arising from element stability.

\subsection{Implementation Details}
\label{sec:impl notes}

We first detail how the Bernstein basis may be used to construct a basis for $\dive \bdd{X}_I(K)$, summarizing the construction in \cite[\S 6.1.1]{AinCP19StokesIII}. Let $K \in \mathcal{T}$ and let $\{B_{\alpha}^{k}\}_{\alpha \in \mathcal{I}^k}$ denote the Bernstein polynomials on $K$:
\begin{align*}
	B_{\alpha}^{k} = \frac{k!}{\alpha_1! \alpha_2! \alpha_3!} \lambda_1^{\alpha_1} \lambda_2^{\alpha_2} \lambda_3^{\alpha_3},
\end{align*}
where $\mathcal{I}^k := \{ \alpha \in \mathbb{Z}_+^3 : |\alpha| = k \}$ and $\{ \lambda_i \}_{i=1}^{3}$ are the barycentric coordinates on $K$. The set $\{ B_{\alpha}^{k} \}_{\alpha \in \mathcal{I}_0^{k}}$, where $\mathcal{I}_0^{k} := \{ \alpha \in \mathcal{I} : \alpha_i < k, \ 1 \leq i \leq 3 \}$ then consists of all degree $k$ polynomials that vanish at the vertices of $K$. Fix any $\gamma \in \mathcal{I}_0^{p-1}$; then, the set $\{ B^{p-1}_{\alpha} - B^{p-1}_{\gamma} : \alpha \in \mathcal{I}_0^{p-1} \setminus \{ \gamma \} \}$ is a basis for $\dive \bdd{X}_I(K)$ thanks to \cref{eq:divxi character} since all Bernstein polynomials have the same average value. As we are using the Bernstein basis for $X_D$ as well, we can then use the algorithms in \cite{Ain11} to compute the element matrices $\bdd{E}_K$, $\bdd{G}_K$, and $\bdd{C}_K$ in $\mathcal{O}(p^4)$ operations and the element load vector $\vec{L}_K$ in $\mathcal{O}(p^3)$ operations. 

For consistency across different flow problems, we invert the sparse matrix $\tilde{\bdd{A}}$ in line 2 of \cref{alg:itpen matrix} using the \texttt{SparseLU} solver in Eigen \cite{eigenweb}, while all local element matrices are inverted using Eigen's \texttt{FullPivLU} solver.

\subsection{Kovasznay Flow}

We first consider Oseen flow \cref{eq:oseen flow} on the rectangular domain $\Omega = (-0.5, 2) \times (-0.5, 1.5)$ with viscosity $\nu = 10^{-1}$, $\bdd{f} = \bdd{0}$, and 
\begin{align*}
	\bdd{w}(x, y) = \begin{bmatrix}
		1 - e^{\kappa x} \cos (2 \pi y) \\
		\frac{\kappa}{2\pi} e^{\kappa x} \sin(2 \pi y)
	\end{bmatrix}, 
\end{align*}
where $\kappa = \frac{1}{2\nu} - \sqrt{\frac{1}{4\nu^2} + 4\pi^2}$. We additionally impose $\bdd{u} = \bdd{w}$ on $\Gamma$. The exact solution to this problem, originally derived by Kovasznay \cite{Kovasznay48} in the context of Navier-Stokes flow, is
\begin{align}
	\label{eq:kov exact sol}
	\bdd{u}(x, y) = \bdd{w}(x, y) \quad \text{and} \quad q(x, y) = -\frac{1}{2} e^{2 \kappa x} - \bar{q}, \qquad (x, y) \in \Omega,
\end{align}
where $\bar{q}$ is the average value of $-\frac{1}{2} e^{2 \kappa x}$ on $\Omega$.

\begin{figure}[htb]
	\centering
	\begin{subfigure}[b]{0.48\linewidth}
		\centering
		\includegraphics[width=0.85\linewidth]{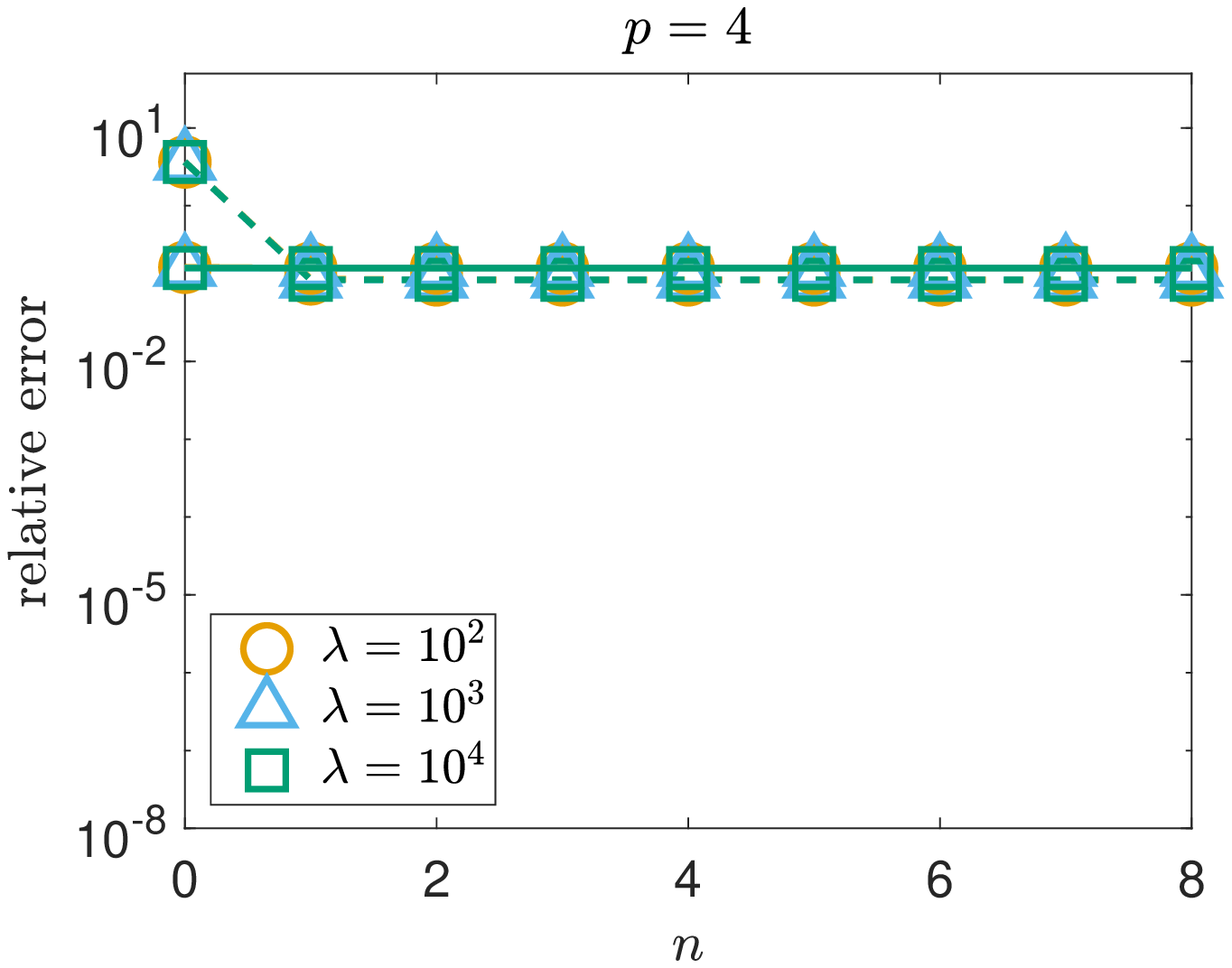}
		\caption{\label{fig:kov iter info p4}}
	\end{subfigure}
	\hfill
	\begin{subfigure}[b]{0.48\linewidth}
		\centering
		\includegraphics[width=0.85\linewidth]{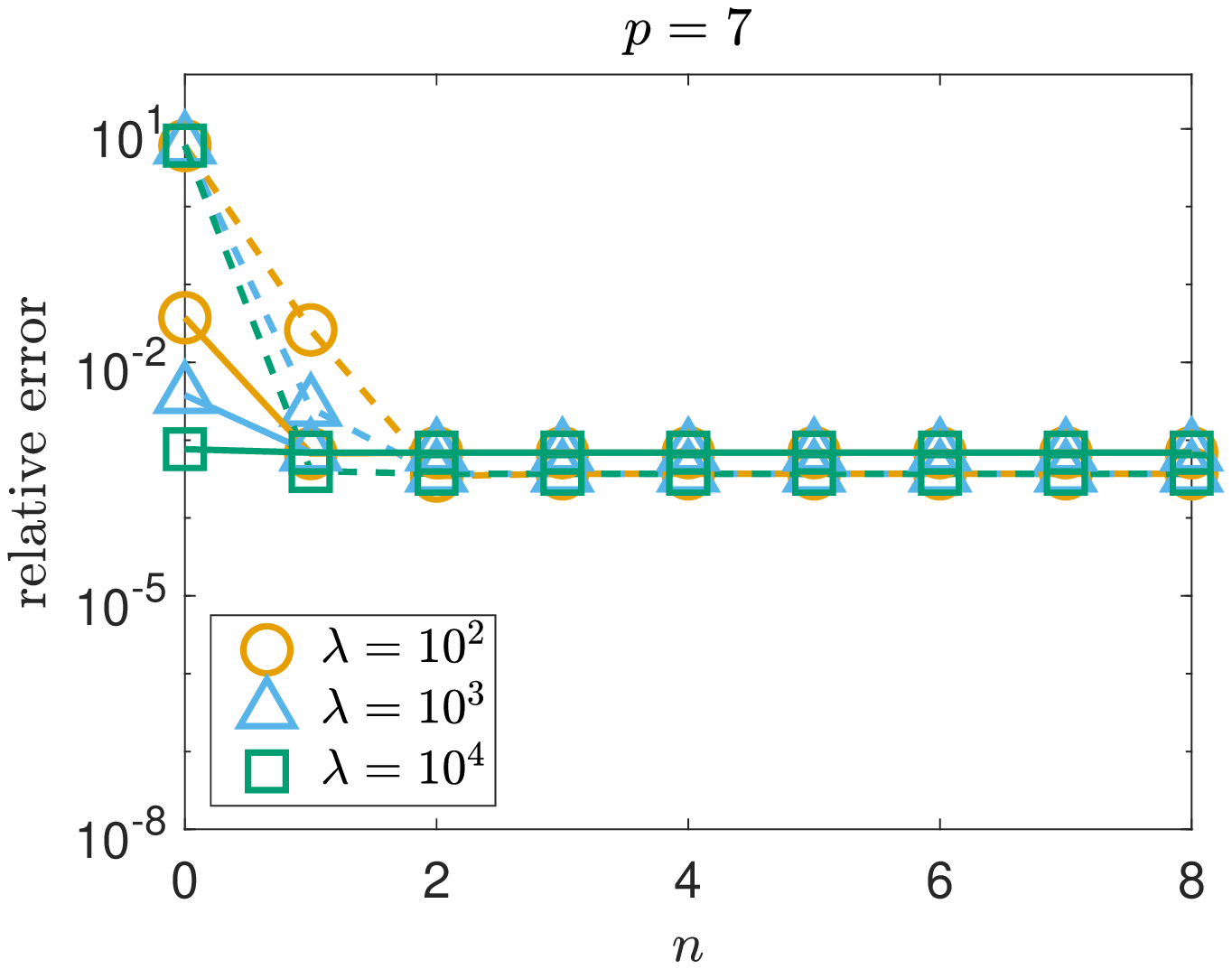}
		\caption{\label{fig:kov iter info p7}}
	\end{subfigure}
	\\
	\begin{subfigure}[b]{0.48\linewidth}
		\centering
		\includegraphics[width=0.85\linewidth]{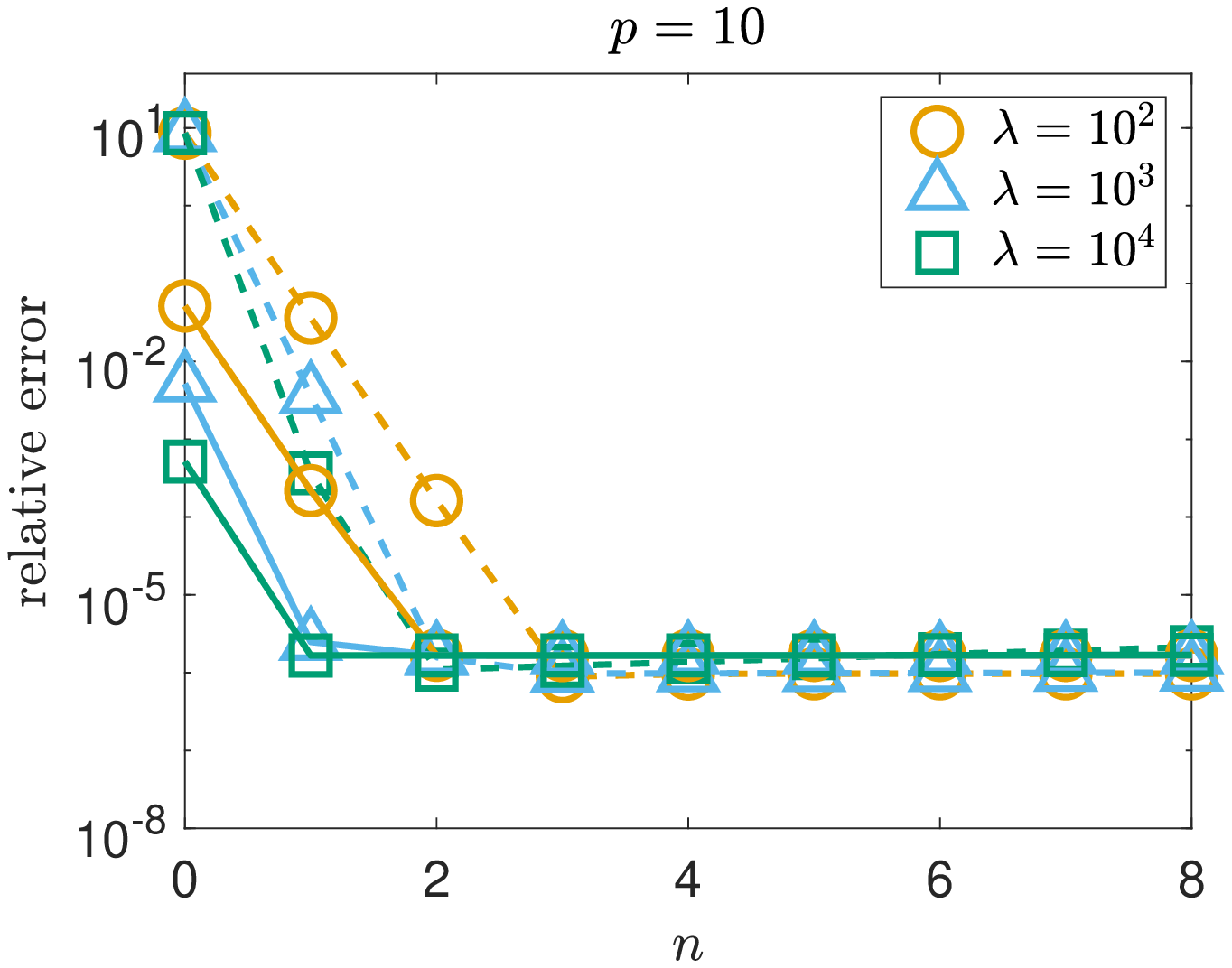}
		\caption{\label{fig:kov iter info p10}}
	\end{subfigure}
	\hfill
	\begin{subfigure}[b]{0.48\linewidth}
		\centering
		\includegraphics[width=0.85\linewidth]{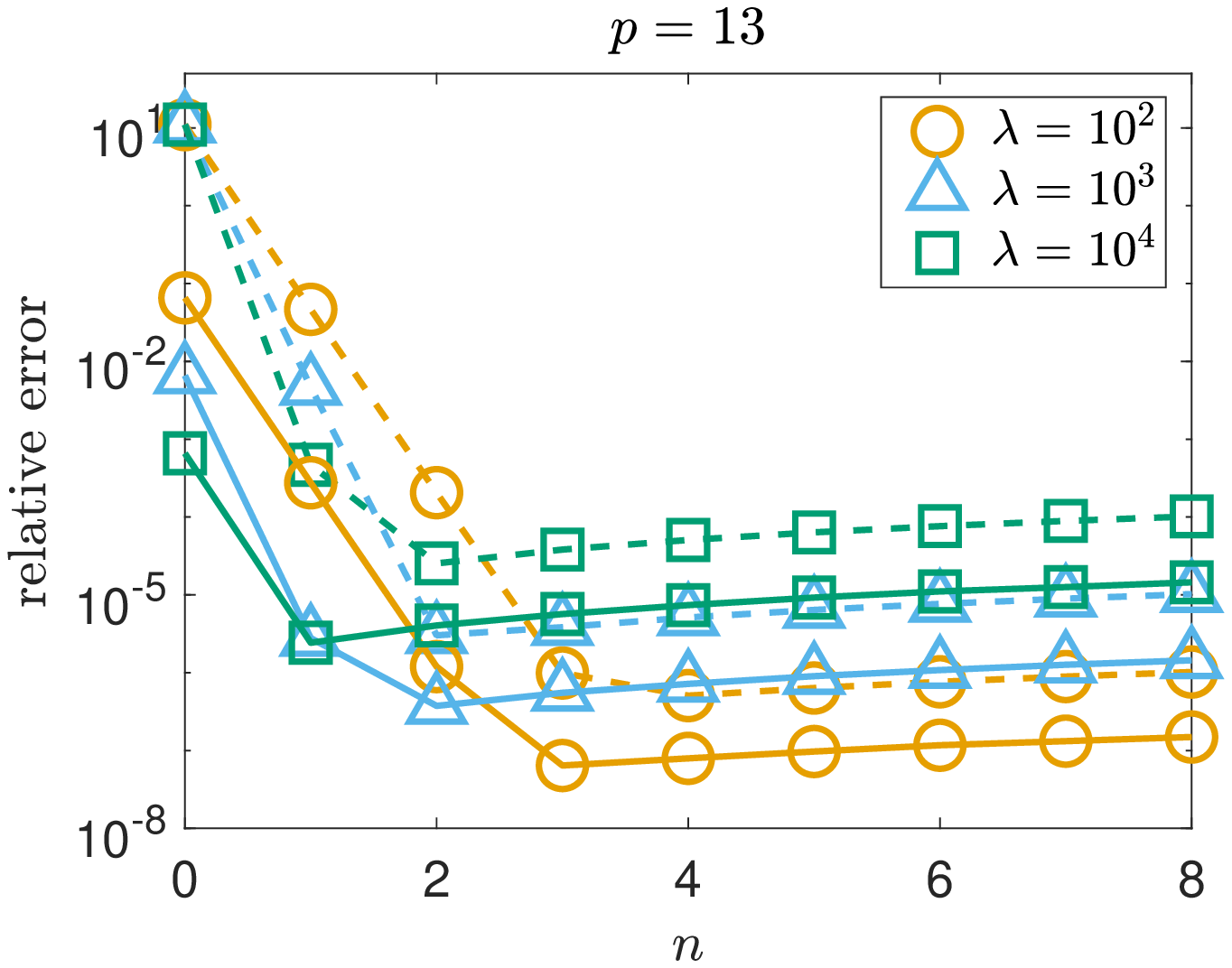}
		\caption{\label{fig:kov iter info error p13}}
	\end{subfigure}
	\caption{Relative velocity error (solid lines) and pressure error (dashed lines) of the solution $(\bdd{u}_X^n, p_X^n)$ given by the SCIP method with (a) $p=4$, (b) $p=7$, (c) $p=10$, and (d) $p=13$ applied to the Kovasznay flow problem with $\nu=10^{-1}$. \label{fig:kov iter erro info}}
\end{figure}

We begin by examining the performance of SCIP using the 4x4 criss-cross mesh in \cref{fig:kov p10}. For $p \in \{4, 7, 10, 13\}$, $\lambda \in \{ 10^2, 10^3, 10^4 \}$, and $0 \leq n \leq 8$, we terminate SCIP after $n$ steps and display the relative velocity error $\|\bdd{u} - \bdd{u}_X^n\|_{1} / \|\bdd{u}\|_1$ and relative pressure error $\|q - q_X^n\| / \|q\|$ in \cref{fig:kov iter erro info}. The relative errors are in agreement with \cref{thm:boundary ip convergence}. The errors decrease until the error in the SCIP method is smaller than the discretization error, at which point the errors level off. Additionally, the pressure errors generally require one to two more iterations of SCIP to level off compared to the velocity errors.

\Cref{fig:kov iter info p} shows the behavior of the velocity and pressure errors versus the polynomial degree on a log-linear scale so that a straight line corresponds to the expected exponential convergence in $p$ since the exact solution \cref{eq:kov exact sol} is analytic \cite{Schwab98}. Observe that, while we indeed see exponential convergence for $p \in \{ 1,\ldots, 10 \}$, for higher values of $p$ there is a loss in accuracy which we attribute to the conditioning of the Bernstein basis.

\begin{figure}[htb]
	\centering
	\begin{subfigure}[b]{0.45\linewidth}
			\includegraphics[width=\linewidth]{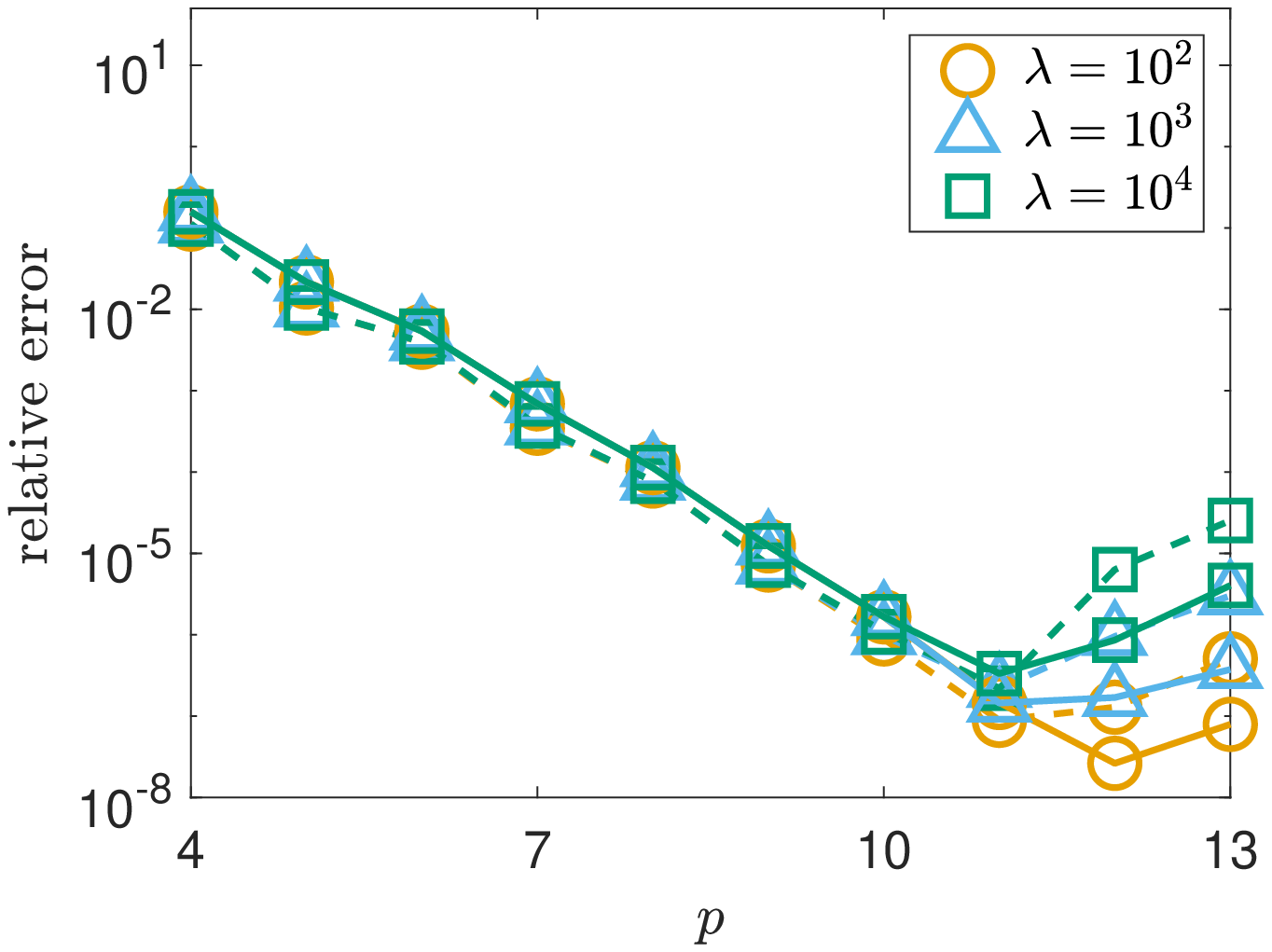}
			\caption{\label{fig:kov iter info p}}
	\end{subfigure}
	\hfill
	\begin{subfigure}[b]{0.53\linewidth}
		\centering
		\includegraphics[height=4.725cm]{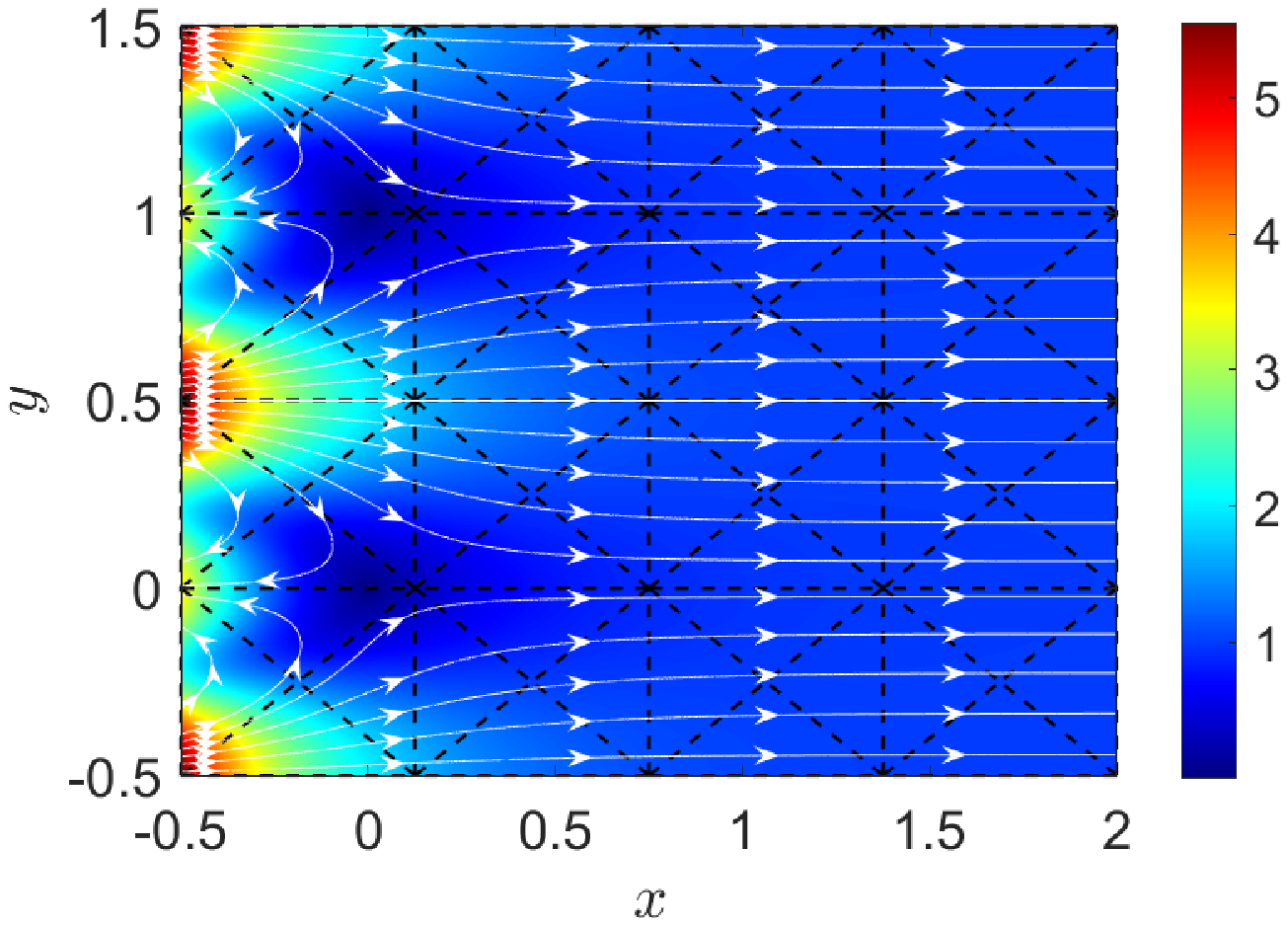}
		\caption{\label{fig:kov p10}}
	\end{subfigure}
	\caption{SCIP approximation to the Kovasznay flow problem with $\nu = 10^{-1}$. (a) Smallest relative velocity error (solid lines) and pressure error (dashed lines) of the solution $(\bdd{u}_X^n, p_X^n)$ over 8 iterations and (b) 4x4 criss-cross mesh (dashed lines) and velocity streamlines (solid lines) with $|\bdd{u}_X^n|$ background color for the $p = 10$ and $\lambda = 10^{3}$ approximation after 8 iterations.}
\end{figure}

The divergence of the SCIP approximation $\| \dive \bdd{u}_X^n \|$ is another important quantity that, according to \cref{thm:boundary ip convergence}, converges exponentially fast as the number of iterations increases. The values of $\| \dive \bdd{u}_X^n \|$ for the same values of $n$ and $p$ in \cref{fig:kov iter erro info} are displayed in \cref{fig:kov iter info}, where in agreement with \cref{eq:scip div convergence}, $\|\dive \bdd{u}_X^n\|$, and hence $\| \dive \tilde{\bdd{u}}^n \|$, decays exponentially fast in $n$, and the rate of decay is greater for larger values of $\lambda$. We observe some degradation of the results when $p > 10$, which we again attribute to roundoff issues with the Bernstein basis. The approximation obtained after 8 iterations of SCIP with $p=10$ and $\lambda=10^{3}$ is displayed in \cref{fig:kov p10}.

\begin{figure}[htb]
	\centering
	\begin{subfigure}[b]{0.48\linewidth}
		\centering
		\includegraphics[width=0.85\linewidth]{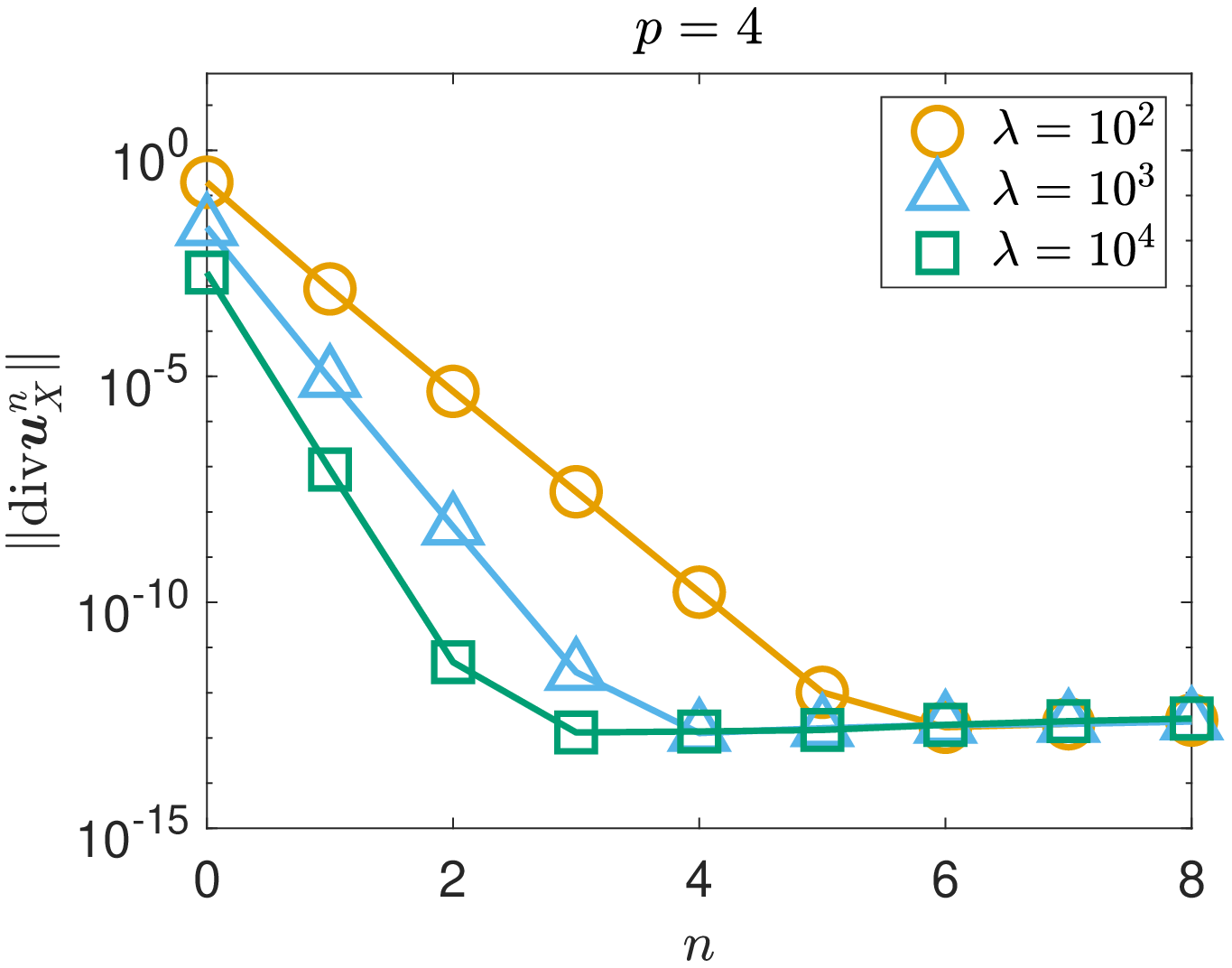}
		\caption{\label{fig:kov iter info div p4}}
	\end{subfigure}
	\hfill
	\begin{subfigure}[b]{0.48\linewidth}
		\centering
		\includegraphics[width=0.85\linewidth]{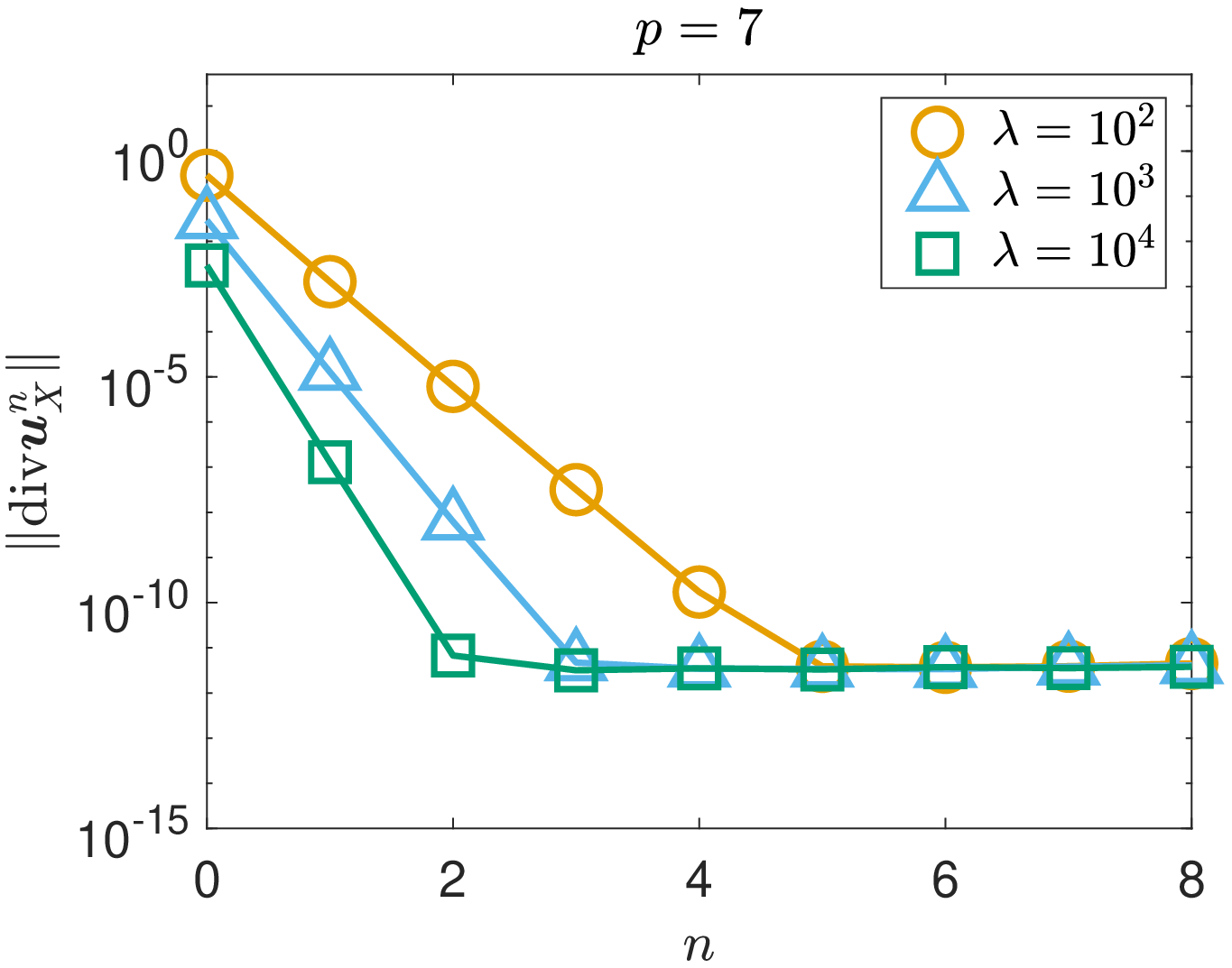}
		\caption{\label{fig:kov iter info div p7}}
	\end{subfigure}
	\\
	\begin{subfigure}[b]{0.48\linewidth}
		\centering
		\includegraphics[width=0.85\linewidth]{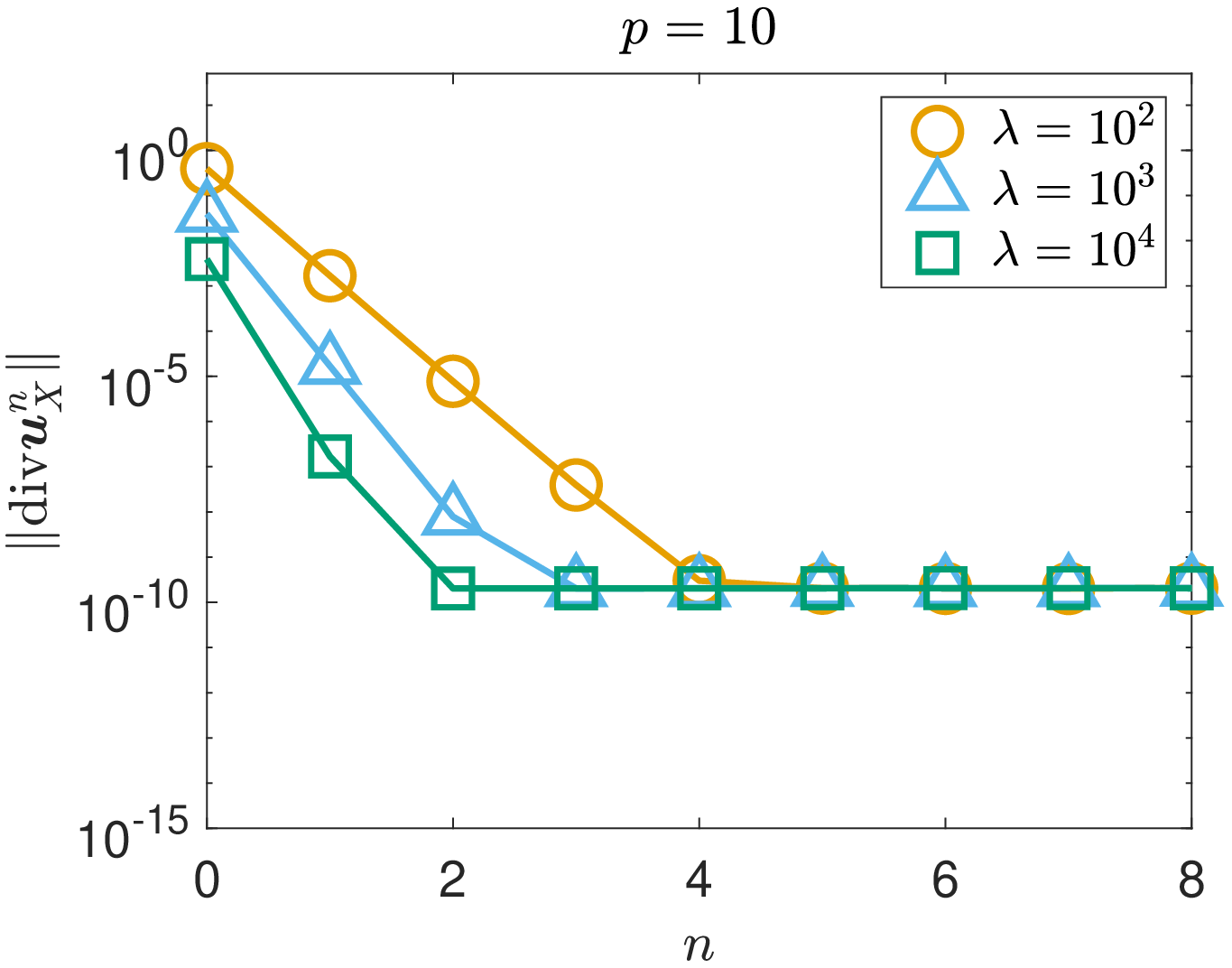}
		\caption{\label{fig:kov iter info div p10}}
	\end{subfigure}
	\hfill
	\begin{subfigure}[b]{0.48\linewidth}
		\centering
		\includegraphics[width=0.85\linewidth]{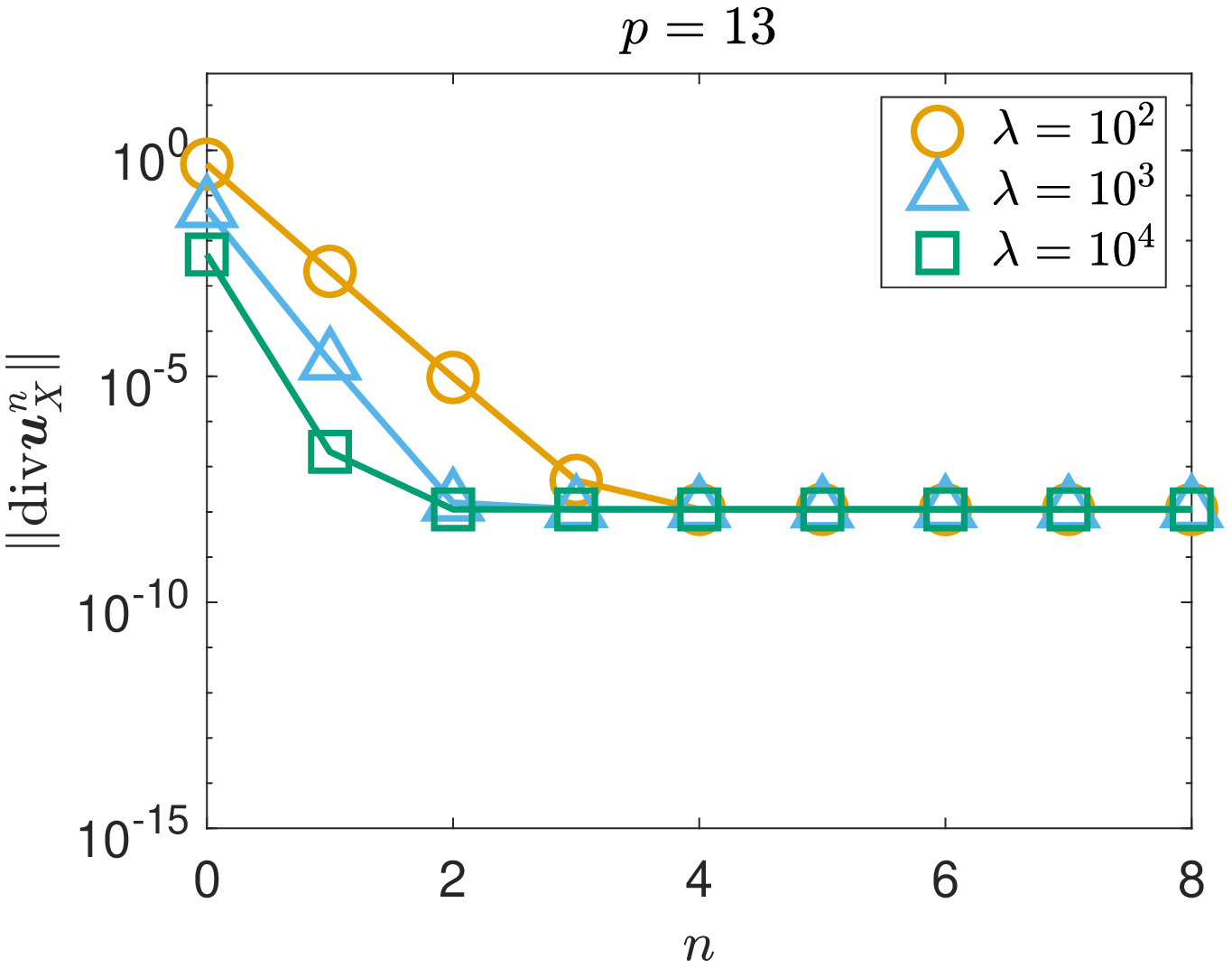}
		\caption{\label{fig:kov iter info div p13}}
	\end{subfigure}
	\caption{Values of $\|\dive \bdd{u}_X^n\|$ with(a) $p=4$, (b) $p=7$, (c) $p=10$, and (d) $p=13$ from the SCIP method applied to the Kovasznay flow problem with $\nu=10^{-1}$. \label{fig:kov iter info}}
\end{figure}

\subsection{Moffatt Eddies}

We now consider an example of Stokes flow due to Moffatt \cite{Moffatt64}, which is a common benchmark for high order methods as it contains features on many scales. Let $\Omega$ be the wedge with a fixed mesh as shown in \cref{fig:moffatt mesh} with the following boundary conditions:
\begin{align*}
	\bdd{u}(x, 0) = \begin{bmatrix}
		1 - x^2 \\ 0
	\end{bmatrix}, \qquad -1 \leq x \leq 1, \quad \text{and} \quad \bdd{u} = \bdd{0} \text{ on } \Gamma \setminus (-1, 1) \times \{0\}.
\end{align*}
The velocity contains an infinite cascade of eddies, each of which is about 400 times weaker than the previous one, while the pressure has an infinite cascade of singularities, starting at $(\pm 1, 0)$. The combination of these two features makes this a challenging test problem.

\begin{figure}[htb]
	\centering
	\begin{subfigure}[b]{0.48\linewidth}
		\centering
		\begin{tikzpicture}[boximg]
			\node[anchor=south west] (zoom0)
			{\includegraphics[width=0.9\linewidth]{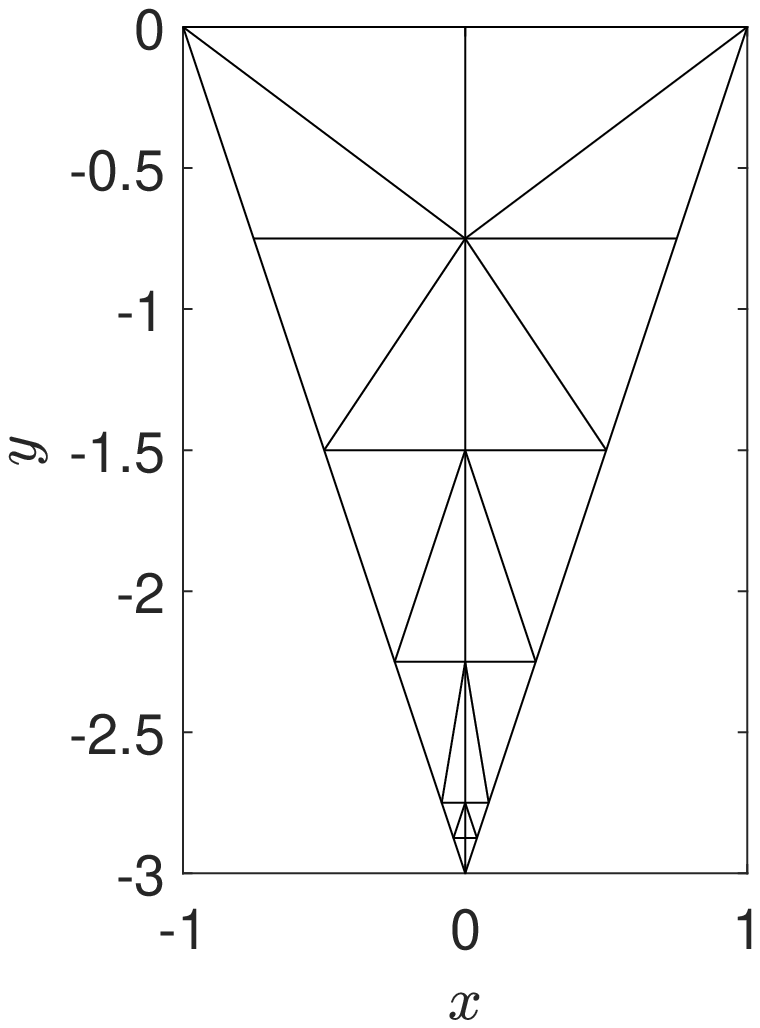}};
			\begin{scope}[x=(zoom0.south east),y=(zoom0.north west)]
				\node[draw,minimum height=0.27cm,minimum width=.4cm] (B0) at (0.5175,0.18) {};
			\end{scope}
		\end{tikzpicture}
		\caption{\label{fig:moffatt mesh}}
	\end{subfigure}
	\hfill
	\begin{subfigure}[b]{0.48\linewidth}
		\centering
		\includegraphics[width=0.9\linewidth]{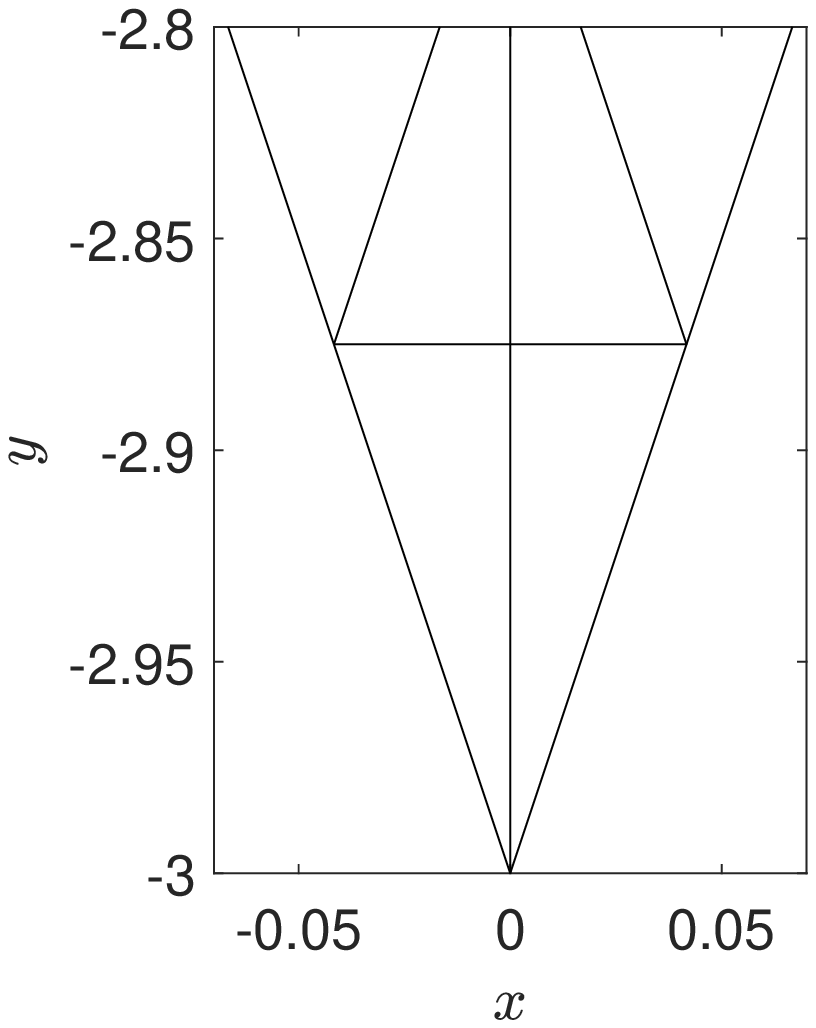}
		\caption{\label{fig:moffatt mesh zoom}}
	\end{subfigure}
	\caption{(a) Computational mesh consisting of 22 elements and (b) zoom for the Moffatt problem.}
\end{figure}

The numerical solution obtained after 8 iterations of the SCIP method with $p=10$ and $\lambda=10^{3}$ on the computational mesh in \cref{fig:moffatt mesh} satisfies  $\|\dive \bdd{u}_{X}^n\| =$ 6.8e-11 and is shown in \cref{fig:moffatt p10}. 
Observe that the method nicely captures the profile of the pressure, as well as the three eddies. In \cref{fig:moffatt p10 zoom}, we zoom in on the numerical solution and observe that an additional two eddies are resolved, with $|\bdd{u}_X^n|$ being on the order of $10^{-11}$. Thus, the method is able to resolve all eddies up to the order of $\|\dive \bdd{u}_{X}^n\|$ and capture the pressure profile without the need to use a priori knowledge of the solution.

\begin{figure}[htb]
	\centering
	\begin{subfigure}[b]{0.48\linewidth}
		\centering
		\includegraphics[width=\linewidth]{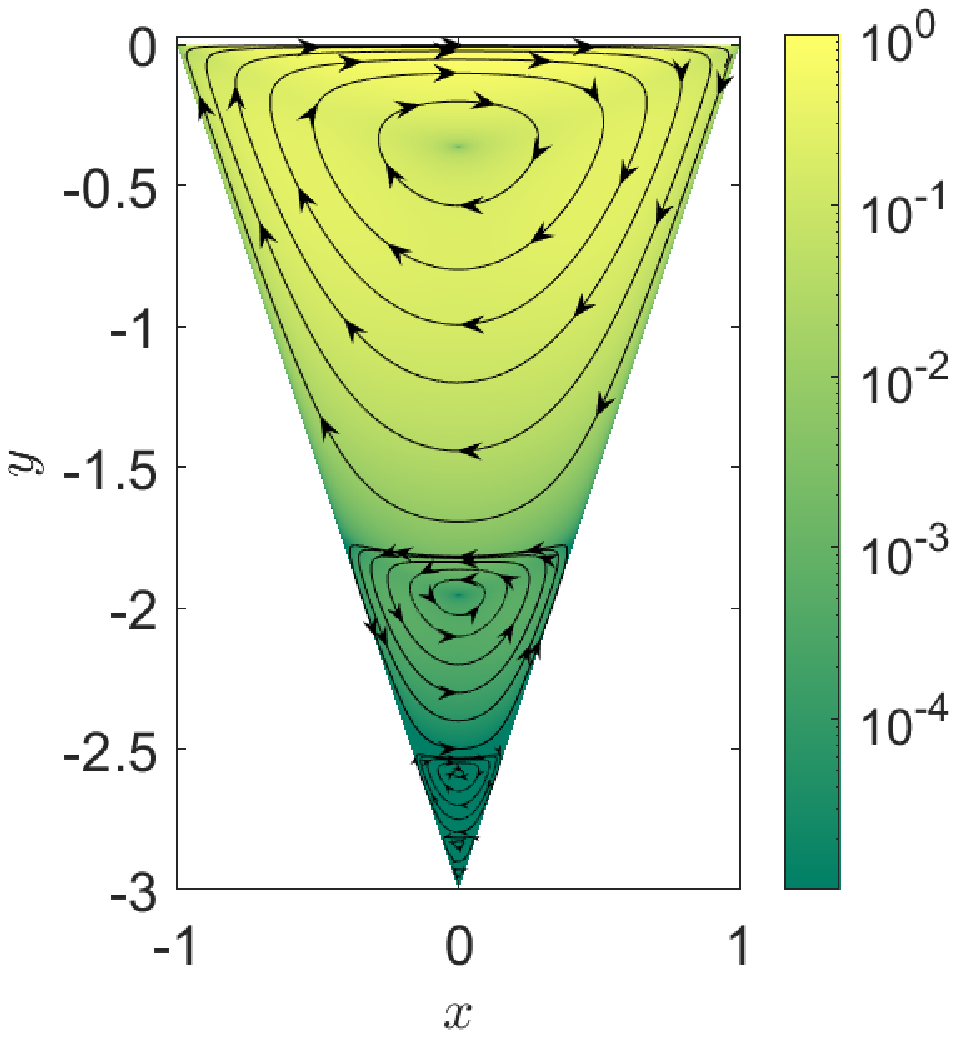}
		\caption{\label{fig:moffat p10 velocity}}
	\end{subfigure}
	\hfill
	\begin{subfigure}[b]{0.48\linewidth}
		\centering
		\includegraphics[width=\linewidth]{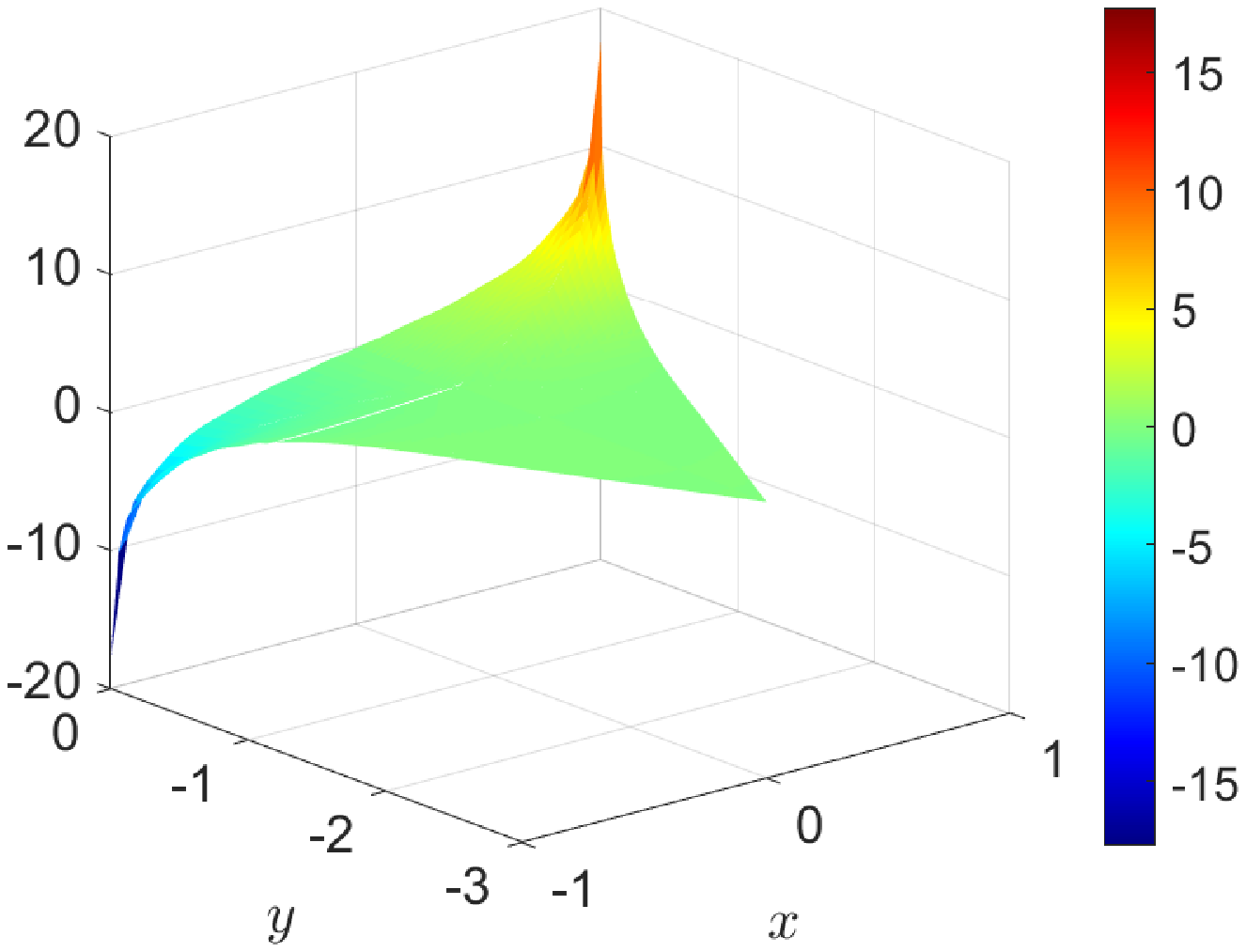}
		\caption{\label{fig:moffat p10 pressure}}
	\end{subfigure}
	\caption{SCIP approximation of the Moffatt problem with $p = 10$ and $\lambda = 10^{3}$: (a) velocity streamlines with $|\bdd{u}|$ background color and (b) pressure. \label{fig:moffatt p10}}
\end{figure}

\begin{figure}[htb]
	\centering
	\begin{subfigure}[b]{0.48\linewidth}
		\centering
		\includegraphics[width=\linewidth]{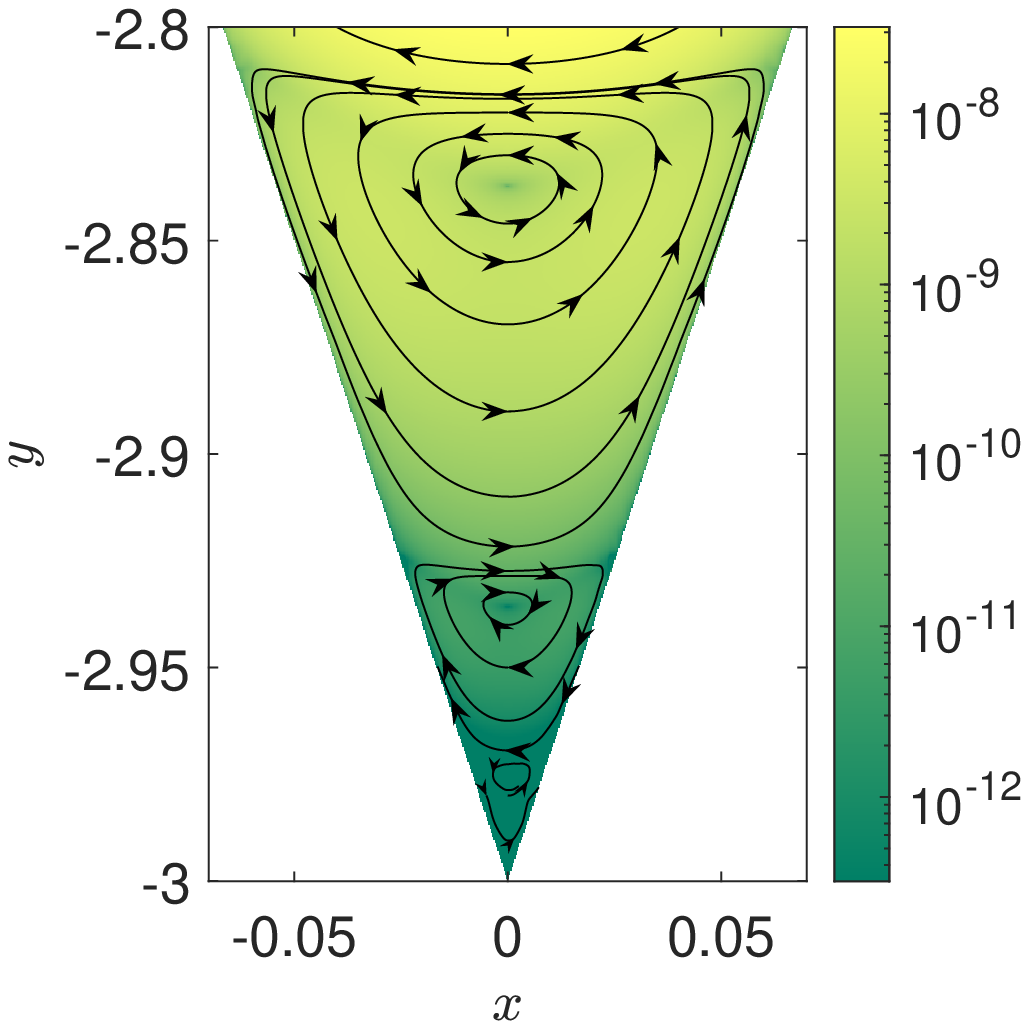}
		\caption{\label{fig:moffat p10 velocity zoom}}
	\end{subfigure}
	\hfill
	\begin{subfigure}[b]{0.48\linewidth}
		\centering
		\includegraphics[width=\linewidth]{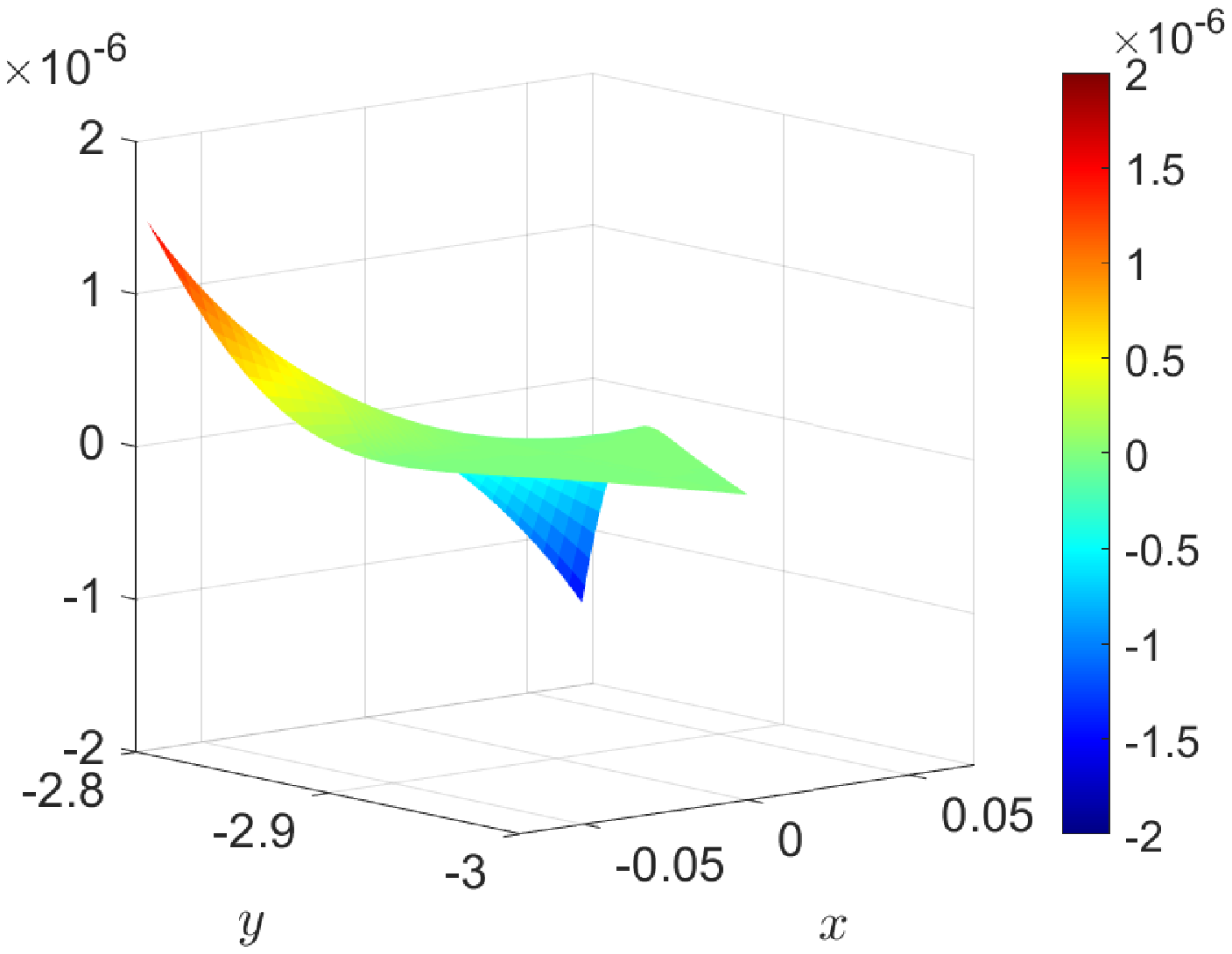}
		\caption{\label{fig:moffat p10 pressure zoom}}
	\end{subfigure}
	\caption{Zoom on bottom eddies of the SCIP approximation of the Moffatt problem with $p = 10$ and $\lambda = 10^{3}$: (a) velocity streamlines with $|\bdd{u}|$ background color and (b) pressure. \label{fig:moffatt p10 zoom}}
\end{figure}

\section{Stokes Extension Operators}
\label{sec:stokes extension}

Given a function $\bdd{u} \in \bdd{X} := X \times X$ and $K \in \mathcal{T}$, \cref{lem:interior inversion} shows that there exists a unique $(\bdd{u}_{S,K}, q_{S,K}) \in \bm{\mathcal{P}}_p(K) \times \dive \bdd{X}_I(K)$ satisfying
\begin{subequations}
	\label{eq:stokes extension def}
	\begin{alignat}{2}
		\label{eq:stokes extension def 1}
		a_K( \bdd{u}_{S,K}, \bdd{v}) - (q_{S,K}, \dive \bdd{v})_K &= 0 \qquad & &\forall \bdd{v} \in \bdd{X}_I(K), \\
		\label{eq:stokes extension def 2}
		-(r, \dive \bdd{u}_{S,K})_K &= 0 \qquad & &\forall r \in \dive \bdd{X}_I(K), \\
		\label{eq:stokes extension def 3}
		( \bdd{u}_{S,K} - \bdd{u} )|_{\partial K} &= \bdd{0}. \qquad & &
	\end{alignat}
\end{subequations}
We define the discrete Stokes extension operators $\mathbb{S} : \bdd{X} \to \bdd{X}$ and $\mathbb{Q} : \bdd{X} \to \dive \bdd{X}_I$ by the rules $\mathbb{S} \bdd{u}|_{K} := \bdd{u}_{S,K}$ and $\mathbb{Q} \bdd{u}|_{K} = q_{S,K}$ for all $K \in \mathcal{T}$. Similarly, there exist $\bdd{u}_{S,K}^{\dagger} \in \bm{\mathcal{P}}_p(K)$ and $q_{S,K}^{\dagger} \in \dive \bdd{X}_I(K)$ satisfying 
\begin{align}
	\label{eq:stokes extension def 1 dagger}
	a_K( \bdd{v}, \bdd{u}_{S,K}^{\dagger}) - (q_{S,K}^{\dagger}, \dive \bdd{v})_K &= 0 \qquad \forall \bdd{v} \in \bdd{X}_I(K),
\end{align}
along with \cref{eq:stokes extension def 2}, and \cref{eq:stokes extension def 3}. We define the ``adjoint" Stokes extension operators $\mathbb{S}^{\dagger}$ and $\mathbb{Q}^{\dagger}$ in terms of $\bdd{u}_{S,K}^{\dagger}$ and $q_{S,K}^{\dagger}$ analogously. 

The next result gives a precise statement of the sense in which the above operators are ``adjoints." Let $s(\cdot, \cdot;\cdot,\cdot)$ denote the Stokes bilinear form
\begin{align*}
	s(\bdd{u}, q; \bdd{v}, r) := a(\bdd{u}, \bdd{v}) - (q, \dive \bdd{v}) - (r, \dive \bdd{u}) \qquad \forall \bdd{u},\bdd{v} \in \bdd{X}, \ \forall q, r \in \dive \bdd{X}.
\end{align*}
Additionally, let $\Pi_I : L^2(\Omega) \to \dive \bdd{X}_I$ denote the usual $L^2(\Omega)$ projection operator onto $\dive \bdd{X}_I$:
\begin{align*}
	(\Pi_I q, r) = (q, r) \qquad \forall q \in L^2(\Omega), \ \forall r \in \dive \bdd{X}_I,
\end{align*}
and $\Pi_I^{\perp} := I - \Pi_I$. Then, we have the following result:
\begin{lemma}
	For all $\bdd{u}, \bdd{v} \in \bdd{X}$ and $q,r \in \dive \bdd{X}$, there holds
	\begin{align}
		\label{eq:stokes adjoint id}
		s(\mathbb{S} \bdd{u}, \mathbb{Q} \bdd{u} + \Pi_I^{\perp} q; \bdd{v}, r) = s(\bdd{u}, q; \mathbb{S}^{\dagger} \bdd{v}, \mathbb{Q}^{\dagger} \bdd{v} + \Pi_I^{\perp} r)
	\end{align}
	and
	\begin{align}
		\label{eq:app:div dagger id}
		\dive \mathbb{S} \bdd{u} = \dive \mathbb{S}^{\dagger} \bdd{u} = \Pi_I^{\perp} \dive \bdd{u}.
	\end{align}
\end{lemma}
\begin{proof}
	Let $\bdd{u} \in \bdd{X}$ be given.	Then, $\bdd{u}_I := \bdd{u} - \mathbb{S} \bdd{u}$ satisfies $\bdd{u}_I \in \bdd{X}_I$ by \cref{eq:stokes extension def 3}, and so $\dive \mathbb{S} \bdd{u} = \dive \bdd{u} + \dive \bdd{u}_I$. Applying $\Pi_I^{\perp}$  gives
	\begin{align*}
		\dive \mathbb{S} \bdd{u} = \Pi_I^{\perp} \dive \mathbb{S} \bdd{u} = \Pi_I^{\perp} \dive \bdd{u} + \Pi_I^{\perp} \dive \bdd{u}_I = \Pi_I^{\perp} \dive \bdd{u},
	\end{align*}
	where we used \cref{eq:stokes extension def 2} and that $\Pi_I \dive \bdd{u}_I = \dive \bdd{u}_I$. Similar arguments show that $\dive \mathbb{S}^{\dagger} \bdd{u} = \Pi_I^{\perp} \dive \bdd{u}$, and \cref{eq:app:div dagger id} follows.
	
	Now let $\bdd{u}, \bdd{v} \in \bdd{X}$ and $q, r \in \dive \bdd{X}$ be given. Thanks to \cref{eq:app:div dagger id}, there holds
	\begin{align*}
		(\Pi_I^{\perp} q, \dive \bdd{v}) + (r, \dive \mathbb{S} \bdd{u}) &=  (q, \Pi_I^{\perp}\dive \bdd{v}) + (r, \Pi_I^{\perp} \dive \bdd{u})\\
		&= (q, \dive \mathbb{S}^{\dagger} \bdd{v}) +  (\Pi_I^{\perp} r, \dive \bdd{u}),
	\end{align*}
	and so
	\begin{multline*}
		s(\mathbb{S} \bdd{u}, \mathbb{Q} \bdd{u} + \Pi_I^{\perp} q; \bdd{v}, r) = a(\mathbb{S} \bdd{u}, \bdd{v}) - (\mathbb{Q} \bdd{u}, \dive \bdd{v}) - (q, \dive \mathbb{S}^{\dagger} \bdd{v})  -  (\Pi_I^{\perp} r, \dive \bdd{u}).
	\end{multline*}
	Since $\bdd{v} - \mathbb{S}^{\dagger} \bdd{v} \in \bdd{X}_I$ by \cref{eq:stokes extension def 3}, we have
	\begin{align*}
		a(\mathbb{S} \bdd{u}, \bdd{v}) - (\mathbb{Q} \bdd{u}, \dive \bdd{v}) = a(\mathbb{S} \bdd{u}, \mathbb{S}^{\dagger} \bdd{v}) - (\mathbb{Q} \bdd{u}, \dive \mathbb{S}^{\dagger} \bdd{v}) = a(\mathbb{S} \bdd{u}, \mathbb{S}^{\dagger} \bdd{v}),
	\end{align*} 
	where we used \cref{eq:stokes extension def 1} and \cref{eq:stokes extension def 2}. Applying similar arguments to $\bdd{u} - \mathbb{S} \bdd{u}$ gives
	\begin{align*}
		a(\mathbb{S} \bdd{u}, \mathbb{S}^{\dagger} \bdd{v}) = a(\bdd{u}, \mathbb{S}^{\dagger} \bdd{v}) + (\mathbb{Q}^{\dagger} \bdd{v}, \dive (\mathbb{S} \bdd{u} - \bdd{u})) = a(\bdd{u}, \mathbb{S}^{\dagger} \bdd{v}) - (\mathbb{Q}^{\dagger} \bdd{v}, \dive \bdd{u}),
	\end{align*}
	where we used \cref{eq:stokes extension def 1 dagger} and \cref{eq:stokes extension def 2}. \Cref{eq:stokes adjoint id} now follows on collecting results.
\end{proof}
\noindent 
The next result characterizes $\tilde{\bdd{X}}_B$ as an invariant subspace of $\bdd{X}_D$ under the operator $\mathbb{S}$ and likewise for $\tilde{\bdd{X}}_B^{\dagger}$ and $\mathbb{S}^{\dagger}$:
\begin{lemma}
	\label{lem:tildexb stokes ext}
	The following identities holds:
	\begin{align}
		\label{eq:tildexb stokes ext}
		\tilde{\bdd{X}}_B &= \{ \bdd{v} \in \bdd{X}_D : \mathbb{S} \bdd{v} = \bdd{v} \} = \{  \mathbb{S} \bdd{v} : \bdd{v} \in \bdd{X}_D \}, \\
		\label{eq:tildexb stokes ext dagger}
		\tilde{\bdd{X}}_B^{\dagger} &= \{ \bdd{v} \in \bdd{X}_D : \mathbb{S}^{\dagger} \bdd{v} = \bdd{v} \}  = \{  \mathbb{S}^{\dagger} \bdd{v} : \bdd{v} \in \bdd{X}_D \}.
	\end{align}
	Moreover,
	\begin{align}
		\label{eq:tilde dive equiv}
		\dive \tilde{\bdd{X}}_B = \dive \tilde{\bdd{X}}_B^{\dagger}.
	\end{align}
\end{lemma}
\begin{proof}
	Let $\bdd{u} \in \bdd{X}_D$ and $K \in \mathcal{T}$. Using the notation of \cref{sec:modified it}, $\bdd{u}|_{K}$ may be expressed as $\bdd{u}|_{K} = \vec{\Phi}_{B,K}^T \vec{u}_{B,K}  + \vec{\Phi}_{I,K}^T \vec{u}_{I,K}$. By \cref{eq:stokes extension def}, $\mathbb{S} \bdd{u}|_{K} = \vec{\Phi}_{B,K}^T \vec{u}_{B,K}  + \vec{\Phi}_{I,K}^T \vec{u}_{I,K}^{\#}$, where
	\begin{align*}
		\begin{bmatrix}
			\bdd{E}_{II} & \bdd{G}_{ \iota I}^T \\
			\bdd{G}_{\iota I} & \bdd{0}
		\end{bmatrix} \begin{bmatrix}
			\vec{u}_{I,K}^{\#} \\
			*
		\end{bmatrix} = - \begin{bmatrix}
			\bdd{E}_{IB} \\
			\bdd{G}_{\iota B} 
		\end{bmatrix} \vec{u}_{B,K}.
	\end{align*}
	Thus, $ \{ \bdd{v} \in \bdd{X}_D : \mathbb{S} \bdd{v} = \bdd{v} \}  = \{  \mathbb{S} \bdd{v} : \bdd{v} \in \bdd{X}_D \} = \{ \bdd{v} \in \bdd{X}_D : \vec{v}_{I,K} = \bdd{S}_K \vec{v}_{B,K} \ \forall K \in \mathcal{T} \}$, and so \cref{eq:tildexb stokes ext}	follows from \cref{lem:tildexb matrix characterization}. Similar arguments give \cref{eq:tildexb stokes ext dagger}. \Cref{eq:tilde dive equiv} is now a consequence of \cref{eq:app:div dagger id}. 
\end{proof}

\begin{lemma}
	For all $\bdd{u}, \bdd{v} \in \bdd{X}$, there holds
	\begin{align}
		\label{eq:app:a dagger id}
		a(\mathbb{S} \bdd{u}, \mathbb{S}\bdd{v}) = a(\mathbb{S} \bdd{u}, \mathbb{S}^{\dagger} \bdd{v}) = a(\mathbb{S}^{\dagger} \bdd{u}, \mathbb{S}^{\dagger} \bdd{v}).
	\end{align}
\end{lemma}
\begin{proof}
	Let $\bdd{u}, \bdd{v} \in \bdd{X}$. By \cref{eq:app:div dagger id,eq:stokes extension def 3}, $\mathbb{S} \bdd{u} - \mathbb{S}^{\dagger} \bdd{u}, \mathbb{S} \bdd{v} - \mathbb{S}^{\dagger} \bdd{v} \in \bdd{N}_I$, and so
	\begin{align*}
		a(\mathbb{S} \bdd{u}, \mathbb{S} \bdd{v}) &= a(\mathbb{S} \bdd{u}, \mathbb{S}^{\dagger} \bdd{v}) + a(\mathbb{S} \bdd{u}, \mathbb{S} \bdd{v} - \mathbb{S}^{\dagger} \bdd{v}) = a( \mathbb{S} \bdd{u}, \mathbb{S}^{\dagger} \bdd{v}) \\
		a(\mathbb{S} \bdd{u}, \mathbb{S}^{\dagger} \bdd{v}) &= a(\mathbb{S}^{\dagger} \bdd{u}, \mathbb{S}^{\dagger} \bdd{v}) + a(\mathbb{S} \bdd{u} - \mathbb{S}^{\dagger} \bdd{u}, \bdd{v}) = a(\mathbb{S}^{\dagger} \bdd{u}, \mathbb{S}^{\dagger} \bdd{v})
	\end{align*}
	by \cref{eq:stokes extension def 1} with $\bdd{v} = \mathbb{S} \bdd{v} - \mathbb{S}^{\dagger} \bdd{v}$ and \cref{eq:stokes extension def 1 dagger} with $\bdd{v} = \mathbb{S} \bdd{u} - \mathbb{S}^{\dagger} \bdd{u}$.
\end{proof}

\begin{lemma}
	\label{lem:uniqueness tilde sigma a}
	Let $\tilde{\bdd{N}}_B := \{ \bdd{z} \in \tilde{\bdd{X}}_B : \dive \bdd{z} \equiv 0 \}$ and $\tilde{\bdd{N}}_B^{\dagger} := \{ \bdd{z} \in \tilde{\bdd{X}}_B^{\dagger} : \dive \bdd{z} \equiv 0 \}$. The variational problem
	\begin{align}
		\label{eq:app:tilde sigma a solve}
		\bdd{z} \in \tilde{\bdd{N}}_B : \qquad a(\bdd{z}, \bdd{w}) = F(\bdd{w}) \qquad \forall \bdd{w} \in \tilde{\bdd{N}}_B^{\dagger}
	\end{align}
	is uniquely solvable for all linear functionals $F$ on $\tilde{\bdd{N}}_B^{\dagger}$. 
\end{lemma}	
\begin{proof}
	Since \cref{eq:app:tilde sigma a solve} is equivalent to a square linear system, it suffices to show uniqueness. Suppose that $\bdd{z} \in \tilde{\bdd{N}}_B$ satisfies \cref{eq:app:tilde sigma a solve} with $F \equiv 0$. Choosing  $\bdd{w} = \mathbb{S}^{\dagger} \bdd{z}$ and applying \cref{eq:app:a dagger id} gives $a(\bdd{z}, \bdd{z}) = a(\bdd{z}, \mathbb{S}^{\dagger} \bdd{z}) = 0$. By ellipticity \cref{eq:a elliptic}, $\bdd{z} \equiv 0$.
\end{proof}

\subsection{Proof of \cref{lem:stokes sol decomp}}
\label{sec:proof decomp}

	Since \cref{eq:stokes tilde spaces} is equivalent to a square linear system, it again suffices to show uniqueness. Let $(\tilde{\bdd{u}}, \tilde{q})$ satisfy \cref{eq:stokes tilde spaces} with zero data on the RHS. Equations \cref{eq:stokes tilde spaces 2} and \cref{eq:tilde dive equiv} means that $\dive \tilde{\bdd{u}} \equiv 0$ and so $\tilde{\bdd{u}} \in \tilde{\bdd{N}}_B$. Choosing $\bdd{v} \in \tilde{\bdd{N}}_B^{\dagger}$ in \cref{eq:stokes tilde spaces 1} shows that $\tilde{\bdd{u}}$ satisfies \cref{eq:app:tilde sigma a solve} with $F \equiv 0$. By \cref{lem:uniqueness tilde sigma a}, $\tilde{\bdd{u}} \equiv 0$.
	
	Thanks to \cref{eq:tilde dive equiv}, there exists $\bdd{v} \in \tilde{\bdd{X}}_B^{\dagger}$ such that $\dive \bdd{v} = \tilde{q}$. Substituting this choice into \cref{eq:stokes tilde spaces 1} gives $\|\tilde{q}\|^2 = (\tilde{q}, \dive \bdd{v}) = 0$, and so $\tilde{q} \equiv 0$. Thus, \cref{eq:stokes tilde spaces} is uniquely solvable. 
	Moreover, for each $K \in \mathcal{T}$, there exists $(\bdd{u}_K, q_K) \in \bdd{X}_I(K) \times \dive \bdd{X}_I(K)$ satisfying \cref{eq:stokes interior} by \cref{lem:interior inversion}.
	
	By \cref{eq:stokes interior}, the functions $\bdd{u}_I := \sum_{K \in \mathcal{T}} \bdd{u}_K$ and $q_I := \sum_{K \in \mathcal{T}} q_K$ satisfy
	\begin{subequations}
		\begin{alignat}{2}
			\label{eq:stokes global interior 1}
			a(\bdd{u}_I, \bdd{v}) - (q_I, \dive \bdd{v}) &= L(\bdd{v}) - a(\tilde{\bdd{u}}, \bdd{v}) \qquad & & \forall \bdd{v} \in \bdd{X}_I, \\
			\label{eq:stokes global interior 2}
			-(r, \dive \bdd{u}_I) &= 0 \qquad & & \forall r \in \dive \bdd{X}_I. 
		\end{alignat}
	\end{subequations}
	Let $\bdd{u}_{X} := \tilde{\bdd{u}} + \bdd{u}_I$ and $q_{X} := \tilde{q} + q_I$. \Cref{eq:stokes tilde spaces 2} means that $\dive \tilde{\bdd{u}} \equiv 0$, while relation \cref{eq:stokes global interior 2} means that $\dive \bdd{u}_I \equiv 0$. As a result, $\dive \bdd{u}_{X} \equiv 0$ and so \cref{eq:stokes variational form fem 2} is satisfied.
	
	We now show that \cref{eq:stokes variational form fem 1} holds. For $\bdd{v} \in \bdd{X}_I$, there holds
	\begin{align*}
		a(\bdd{u}_{X}, \bdd{v}) - (q_{X}, \dive \bdd{v}) &= a(\tilde{\bdd{u}}, \bdd{v}) + a(\bdd{u}_I, \bdd{v}) - (q_I, \dive \bdd{v}) = L(\bdd{v}),
	\end{align*}
	where we used \cref{eq:stokes global interior 1} and that $(\tilde{q}, \dive \bdd{v}) = 0$ by \cref{eq:tildexb def}. For $\bdd{v} \in \tilde{\bdd{X}}_B^{\dagger}$, there holds
	\begin{align*}
		a(\bdd{u}_{X}, \bdd{v}) - (q_{X}, \dive \bdd{v}) = a(\tilde{\bdd{u}}, \bdd{v}) - (\tilde{q}, \dive \bdd{v}) + a(\bdd{u}_I, \bdd{v}) = L(\bdd{v}) + a(\bdd{u}_I, \bdd{v}),
	\end{align*}
	where we used \cref{eq:stokes tilde spaces 1} and that $(q_I, \dive \bdd{v}) = 0$ by \cref{eq:tildexb def}. Since $\bdd{v} \in \tilde{\bdd{X}}_B^{\dagger}$ and $\dive \bdd{u}_I \equiv 0$, $a(\bdd{u}_I, \bdd{v} ) = 0$ by definition. \Cref{eq:stokes variational form fem 1} now follows from linearity thanks to the decomposition \cref{eq:xd decomp} with $\tilde{\bdd{X}}_B^{\dagger}$. \qed

\section{Convergence of SCIP and the Iterated Penalty Method}
\label{sec:scip convergence}

We begin with an estimate for functions in $\tilde{\bdd{X}}_B$ that are orthogonal to divergence free functions:
\begin{lemma}
	\label{lem:app:orthog divefree tilde}
	Let  $\tilde{\bdd{N}}_B^{\perp} := \{ \bdd{v} \in \tilde{\bdd{X}}_B : a(\bdd{v}, \bdd{z}) = 0 \ \forall \bdd{z} \in \tilde{\bdd{N}}_B^{\dagger} \}$. Then, 
	\begin{align}
		\label{eq:app:div dagger perp cont}
		\|\bdd{u}\|_{1} \leq \frac{(M+\alpha)^2}{\alpha^2 \beta_X} \|\dive \bdd{u}\| \qquad \forall \bdd{u} \in \tilde{\bdd{N}}_B^{\perp},	
	\end{align}
	where $M > 0$ \cref{eq:a bounded}, $\alpha > 0$ \cref{eq:a elliptic}, and $\beta_X > 0$ \cref{eq:discrete inf-sup}.
\end{lemma}
\begin{proof}
	Let $\bdd{u} \in \tilde{\bdd{N}}_B^{\perp}$. By the proof of \cref{thm:xd decomp and inf-sup}, there exists $\bdd{v} \in \tilde{\bdd{X}}_B$ such that $\dive \bdd{v} = \dive \bdd{u}$ and $\|\bdd{v}\|_{1} \leq (M+\alpha)(\alpha \beta_X)^{-1} \|\dive \bdd{u}\|$. 
	Since \cref{eq:app:tilde sigma a solve} is uniquely solvable by \cref{lem:uniqueness tilde sigma a}, there exists $\bdd{z} \in \tilde{\bdd{N}}_B$ such that $a(\bdd{z}, \bdd{n}) = a(\bdd{v}, \bdd{n})$ for all $\bdd{n} \in \tilde{\bdd{N}}_B^{\dagger}$. By ellipticity \cref{eq:a elliptic} and \cref{eq:app:a dagger id}, we have
	\begin{align*}
		\alpha \|\bdd{z}\|_1^2 \leq a(\bdd{z}, \bdd{z}) = a(\bdd{z}, \mathbb{S}^{\dagger} \bdd{z}) = a(\bdd{v}, \mathbb{S}^{\dagger} \bdd{z}) = a(\bdd{v}, \bdd{z}) \leq M \|\bdd{v}\|_{1} \|\bdd{z}\|_{1},
	\end{align*}
	since $\mathbb{S} \bdd{z} = \bdd{z}$ by \cref{lem:tildexb stokes ext}. Thus, $\|\bdd{z}\|_1 \leq M \alpha^{-1} \|\bdd{v}\|_1$.
	
	Let $\bdd{w} := \bdd{v} - \bdd{z}$. By construction, $\dive \bdd{w} = \dive \bdd{u}$ and $\bdd{w} \in \tilde{\bdd{N}}_B^{\perp}$. Moreover, $\bdd{w}$ satisfies \cref{eq:app:div dagger perp cont} thanks to the triangle inequality. To complete the proof, we now show that $\bdd{u} = \bdd{w}$. Since $\dive (\bdd{u} - \bdd{w}) = 0$, there exists $\bdd{e} := \bdd{u} - \bdd{w} \in \tilde{\bdd{N}}_B^{\perp}$. Thanks to \cref{eq:app:a dagger id}, $0 = a(\bdd{e}, \mathbb{S}^{\dagger} \bdd{ e}) = a(\bdd{e}, \bdd{e}) \geq \alpha \|\bdd{e}\|_1^2$. Consequently, $\bdd{e} \equiv \bdd{0}$ and so $\bdd{u} = \bdd{w}$.
\end{proof}

With \cref{lem:app:orthog divefree tilde} in hand, the proof of \cref{thm:boundary ip convergence} is a generalization of the convergence proof for the standard iterated penalty method (see e.g. \cite[p.356-359]{Brenner08}).

\begin{proof}[Proof of \cref{thm:boundary ip convergence}]
	Let $\bdd{e}^{n} := \tilde{\bdd{u}} - \tilde{\bdd{u}}^n$ and $r^n := \tilde{q} - \tilde{q}^n$, $n \in \mathbb{N}_0$. Subtracting \cref{eq:boundary ip 1} from \cref{eq:stokes tilde spaces 1} gives, for $n \in \mathbb{N}_0$,
	\begin{align}
		\label{eq:app:scip relations 1}
		a_{\lambda}(\bdd{e}^n, \bdd{v}) &= (r^n, \dive \bdd{v}) \qquad \forall \bdd{v} \in \bdd{X}_B^{\dagger}
	\end{align}
	and $r^{n+1} = r^{n} + \lambda \dive \tilde{\bdd{u}}^n = r^n - \lambda \dive \bdd{e}^n$ since $\dive \tilde{\bdd{u}} = 0$. Using these relations, we obtain
	\begin{align*}
		a_{\lambda}(\bdd{e}^{n+1}, \bdd{v}) &=  (r^n, \dive \bdd{v}) -	\lambda (\dive \bdd{e}^n, \dive \bdd{v})  
		= a_{\lambda}(\bdd{e}^n, \bdd{v}) - \lambda (\dive \bdd{e}^n, \dive \bdd{v}) 
		= a(\bdd{e}^n, \bdd{v})
	\end{align*}	
	for all $\bdd{v} \in \tilde{\bdd{X}}_B^{\dagger}$. Choosing $\bdd{v} = \mathbb{S}^{\dagger} \bdd{e}^{n+1}$ and using \cref{eq:app:a dagger id,eq:app:div dagger id} then gives
	\begin{align}
		\label{eq:app:scip u error relations}
		a(\bdd{e}^{n+1}, \bdd{e}^{n+1}) + \lambda \| \dive \bdd{e}^{n+1}\|^2 = a(\bdd{e}^{n}, \mathbb{S}^{\dagger} \bdd{e}^{n+1}) = a(\bdd{e}^n, \bdd{e}^{n+1}) \leq M \|\bdd{e}^n\|_1 \|\bdd{e}^{n+1}\|_1,
	\end{align}
	where we used \cref{eq:app:a dagger id} and that $\mathbb{S} \bdd{e}^{n+1} = \bdd{e}^{n+1}$. Moreover, \cref{eq:app:scip relations 1} shows that $\bdd{e}^{n+1} \in \tilde{\bdd{N}}_B^{\perp}$ for all $n \in \mathbb{N}_0$. Applying \cref{lem:app:orthog divefree tilde} and \cref{eq:a elliptic} to the LHS of \cref{eq:app:scip u error relations} gives
	\begin{align}
		\label{eq:app:scip u error 1}
		( \alpha + \lambda \tilde{\Upsilon}^{-2} ) \|\bdd{e}^{n+1}\|_1 \leq M \|\bdd{e}^{n}\|_1 \implies 	\| \bdd{e}^{n} \|_{1} \leq ( M \tilde{\Upsilon}^2 \lambda^{-1} )^n \| \bdd{e}^{0} \|_{1},
	\end{align}
	where $\tilde{\Upsilon} := (M+\alpha)^2 / (\alpha^2 \beta_X)$, the constant appearing in \cref{eq:app:div dagger perp cont}. \Cref{eq:scip div convergence} now follows from \cref{eq:app:scip u error 1} on noting that $\dive \tilde{\bdd{u}}^n = -\dive \bdd{e}^n$.
	
	Now, we use \cref{lem:app:orthog divefree tilde} to obtain
	\begin{align}
		\label{eq:app:scip u error 2}
		\|\bdd{e}^n\| \leq \tilde{\Upsilon} \|\dive \bdd{e}^n\| = \tilde{\Upsilon} \|\dive \tilde{\bdd{u}}^n\|.
	\end{align}
	Applying the inf-sup condition \cref{eq:inf-sup boundary} for $\tilde{\bdd{X}}_B^{\dagger} \times \dive \tilde{\bdd{X}}_B^{\dagger}$ and using \cref{eq:app:scip relations 1,eq:tilde dive equiv} gives
	\begin{align*}
		\tilde{\beta}_X \|r^{n}\| \leq \sup_{\bdd{0} \neq \bdd{v} \in \tilde{\bdd{X}}_B^{\dagger}} \frac{(r^n, \dive \bdd{v})}{\|\bdd{v}\|_{1}} =  \sup_{\bdd{0} \neq \bdd{v} \in \tilde{\bdd{X}}_B^{\dagger}} \frac{a_{\lambda}(\bdd{e}^n, \bdd{v}) }{|\bdd{v}|_{1}} \leq M  \| \bdd{e}^{n}\|_1 + \sqrt{d}\lambda \|\dive \bdd{e}^n\|,
	\end{align*}
	where $\tilde{\beta}_X := \alpha \beta_X (M + \alpha)^{-1}$. Thanks to \cref{eq:app:scip u error 2}, $\|r^{n}\| \leq (M \tilde{\Upsilon} + \sqrt{d} \lambda) \tilde{\beta}_X^{-1} \|\dive \tilde{\bdd{u}}^n\|$. \Cref{eq:scip convergence} now follows on collecting results.
\end{proof}

\subsection{Convergence of the Standard IP Method}
\label{sec:vanilla ip convergence}

The following result, which is an immediate consequence of \cite[eq. (13.1.16)]{Brenner08}, is the analogue of \cref{lem:app:orthog divefree tilde}:
\begin{lemma}
	\label{lem:app:orthog divefree}
	For all $\bdd{u} \in \{ \bdd{v} \in \bdd{X}_D : a(\bdd{v}, \bdd{w}) = 0 \ \forall \bdd{w} \in \bdd{X}_D : \dive \bdd{w} = 0 \}$, there holds
	\begin{align*}
		\| \bdd{u} \|_1 \leq \frac{M + \alpha}{\alpha \beta_X} \|\dive \bdd{u} \|,
	\end{align*}
	where $M > 0$ \cref{eq:a bounded}, $\alpha > 0$ \cref{eq:a elliptic}, and $\beta_X > 0$ \cref{eq:discrete inf-sup}.
\end{lemma}
\noindent
With \cref{lem:app:orthog divefree} in hand, the proof of \cref{thm:vanilla ip convergence} is analogous to the proof of \cref{thm:boundary ip convergence}: the spaces $\tilde{\bdd{X}}_B$ and $\tilde{\bdd{X}}_B^{\dagger}$ are replaced by $\bdd{X}_D$; the choice $\bdd{v} = \mathbb{S}^{\dagger} \bdd{e}^{n+1}$ is replaced by $\bdd{v} = \bdd{e}^{n+1}$; the use of \cref{lem:app:orthog divefree tilde} and $\tilde{\Upsilon}$ are replaced by \cref{lem:app:orthog divefree} and $\Upsilon := (M+\alpha)/(\alpha \beta_X)$; and the inf-sup constant $\tilde{\beta}_X$ is replaced by $\beta_X$ defined in \cref{eq:discrete inf-sup}.

\appendix

\section{Properties of the 2D Scott-Vogelius Elements}
\label{sec:proof of optimal approx}

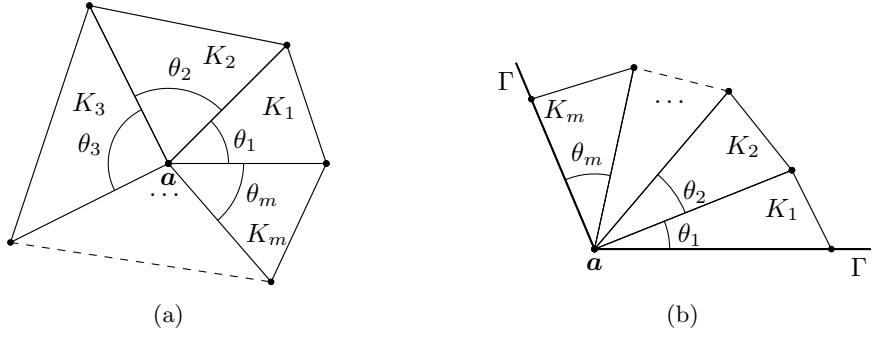
\begin{figure}[htb]
	\centering
	\begin{subfigure}[b]{0.45\linewidth}
		\centering
		\begin{tikzpicture}[scale=0.525]
			
			\coordinate (a) at (0, 0);
			\coordinate (a0) at (4, 0);
			\coordinate (a1) at (3, 3);
			\coordinate (a2) at (-2, 4);
			\coordinate (a3) at (-4, -2);
			\coordinate (a4) at (2.6, -3);				
			
			\filldraw (a) circle (2pt) node[align=center,below]{$\bdd{a}$}
			-- (a0) circle (2pt) 	
			-- (a1) circle (2pt) 
			-- (a);
			\filldraw (a) circle (2pt) node[align=center,below]{}	
			-- (a1) circle (2pt) 
			-- (a2) circle (2pt) 
			-- (a);
			\filldraw (a) circle (2pt) node[align=center,below]{}	
			-- (a2) circle (2pt) 
			-- (a3) circle (2pt) 
			-- (a);
			\filldraw (a) circle (2pt) node[align=center,below]{}	
			-- (a4) circle (2pt) 
			-- (a0) circle (2pt) 
			-- (a);
			\draw[dashed] (a3) -- (a4);
			\draw (2.8, 1.4) node(K0){$K_1$};
			\draw (1.3, 2.7) node(K1){$K_2$};
			\draw (-2, 1.5) node(K1){$K_3$};
			\draw (-0.1, -0.8) node(Kdots){$\ldots$};
			\draw (2.5, -1.8) node(Km){$K_m$};
			\pic["$\theta_1$"{anchor=west}, draw, angle radius=0.8cm, angle eccentricity=1] {angle=a0--a--a1};
			\pic["$\theta_2$"{anchor=south}, draw, angle radius=1cm, angle eccentricity=1] {angle=a1--a--a2};
			\pic["$\theta_3$"{anchor=east}, draw, angle radius=0.8cm, angle eccentricity=1] {angle=a2--a--a3};
			\pic["$\theta_m$"{anchor=west}, draw, angle radius=1cm, angle eccentricity=1] {angle=a4--a--a0};
			
		\end{tikzpicture}
		\caption{}
		\label{fig:internal schema}
	\end{subfigure}
	\hfill
	\begin{subfigure}[b]{0.5\linewidth}
		\centering
		\begin{tikzpicture}[scale=0.525]
			\filldraw (0,0) circle (2pt) node[align=center,below]{$\bdd{a}$}
			-- (6,0) circle (2pt) node[align=center,below]{}	
			-- (5,2) circle (2pt) node[align=center,above]{}
			-- (0,0);
			\filldraw (0,0) circle (2pt) node[align=center,below]{}	
			-- (5,2) circle (2pt) node[align=center,above]{}
			-- (3.4,4) circle (2pt) node[align=center,below]{}
			-- (0,0);
			\filldraw (0,0) circle (2pt) node[align=center,below]{}	
			-- (3.4,4) circle (2pt) node[align=center,above]{};
			\filldraw (1,4.6) circle (2pt) node[align=center,below]{}
			-- (0,0);
			\filldraw (0,0) circle (2pt) node[align=center,below]{}	
			-- (1,4.6) circle (2pt) node[align=center,above]{}
			-- (-1.6,3.8) circle (2pt) node[align=center,below]{}
			-- (0,0);
			
			\draw[dashed] (3.4,4) -- (1,4.6);
			
			\coordinate(a) at (0, 0);
			\coordinate(a0) at (6, 0);
			\coordinate(a1) at (5, 2);
			\coordinate(a2) at (3.4, 4);
			\coordinate(a3) at (1, 4.6);
			\coordinate(a4) at (-1.6, 3.8);
			
			\coordinate (a12) at ($(a)!2/3!(a1)$);
			\coordinate (a121) at ($(a)!2/3-1/sqrt(29)!(a1)$);
			\coordinate (a1205) at ($(a)!2/3-0.5/sqrt(29)!(a1)$);
			
			\coordinate (a22) at ($(a)!2/3!(a2)$);
			\coordinate (a221) at ($(a)!2/3-1/sqrt(27.56)!(a2)$);
			\coordinate (a2205) at ($(a)!2/3-0.5/sqrt(27.56)!(a2)$);
			
			\coordinate (a32) at ($(a)!2/3!(a3)$);
			\coordinate (a321) at ($(a)!2/3-1/sqrt(22.16)!(a3)$);
			\coordinate (a3205) at ($(a)!2/3-0.5/sqrt(22.16)!(a3)$);
			
			\coordinate (a42) at ($(a)!2/3!(a4)$);
			\coordinate (a421) at ($(a)!2/3-1/sqrt(17)!(a4)$);
			\coordinate (a4205) at ($(a)!2/3-0.5/sqrt(17)!(a4)$);

			\pic["$\theta_1$"{anchor=west}, draw, angle radius=1cm, angle eccentricity=1] {angle=a0--a--a1};
			\pic["$\theta_2$"{anchor=west}, draw, angle radius=1.3cm, angle eccentricity=1] {angle=a1--a--a2};
			\pic["$\theta_m$"{anchor=south}, draw, angle radius=1cm, angle eccentricity=1] {angle=a3--a--a4};
			
			\draw (4.75, 1) node(K0){$K_1$};
			\draw (3.75, 2.6) node(K1){$K_2$};
			\draw (1.9, 3.75) node(Kdots){$\ldots$};
			\draw (-0.75, 3.5) node(Km){$K_m$};
			
			\filldraw[thick] (0, 0) -- (7,0);
			\filldraw[thick] (0, 0) -- ($(-1.6, 3.8) + ($(a)!1/sqrt(17)!(a4)$) $);
			
			\draw (6.7, 0) node[align=center,below](G0){$\Gamma$};
			\draw ($(-1.6, 3.8) + ($(a)!0.5/sqrt(17)!(a4)$) $) node[align=center,left](Gm1){$\Gamma$};
			
		\end{tikzpicture}		
		\caption{}
		\label{fig:boundary schema}
	\end{subfigure}
	\caption{Notation for mesh around (a) an internal vertex $\bdd{a} \in \mathcal{V}_I$ and (b) a boundary vertex $\bdd{a} \in \mathcal{V}_D$, each abutting $m = |\mathcal{T}_{\bdd{a}}|$ elements.}
	\label{fig:patch schema}
\end{figure}

Finally, in this section, we turn to the the fundamental stability and approximation properties of the 2D Scott-Vogelius elements, as well as discrete exact sequence properties. One of the key conditions for optimal approximation properties is that the mesh is \textit{corner-split at Dirichlet vertices}. Roughly speaking, a mesh is corner-split if every element has at most one edge lying on $\Gamma_{D}$. In order to give a precise definition, we let $\mathcal{V}$ denote the set of element vertices, $\mathcal{V}_I$ the set of interior vertices, and $\mathcal{V}_C$ the set of element vertices coinciding with the corners of the physical domain $\Omega$. For $\bdd{a} \in \mathcal{V}_I \cup \mathcal{V}_D$, we label the elements as in \cref{fig:patch schema} and define
\begin{align}
	\label{eq:xi def}
	\xi(\bdd{a}) := \sum_{i=1}^{|\mathcal{T}_{\bdd{a}}| - \eta_{\bdd{a}}} |\sin(\theta_i + \theta_{i+1})|, \qquad \text{where } \eta_{\bdd{a}} = \begin{cases}
		1 & \bdd{a} \in \mathcal{V}_D, \\
		0 & \bdd{a} \in \mathcal{V}_I.
	\end{cases}
\end{align}
A mesh is \textit{corner-split at Dirichlet vertices} if $\{ \bdd{a} \in \mathcal{V}_C \cap \mathcal{V}_D : \xi(\bdd{a}) = 0 \} = \emptyset$.

The following result states that the 2D Scott-Vogelius elements are uniformly inf-sup stable in $h$ and $p$ and possess optimal approximation properties under mild assumptions on the mesh:
\begin{theorem}
	\label{thm:optimal approx}
	Suppose that $p \geq 4$ and that the family of meshes $\{ \mathcal{T} \}$ is corner-split at Dirichlet vertices and satisfies \cite[eq. (5.14)]{AinCP21LE}. Then, the Scott-Vogelius elements are uniformly inf-sup stable in $h$ and $p$; i.e., there exists $\beta > 0$ independent of $h$ and $p$ such that
	\begin{align}
		\label{eq:discrete inf-sup 2d}
		\beta \|q\| \leq \sup_{\bdd{0} \neq \bdd{v} \in \bdd{X}_D} \frac{(\dive \bdd{v}, q)}{\|\bdd{v}\|_{1}} \qquad \forall q \in \dive \bdd{X}_D.
	\end{align}
	Moreover, for $\bdd{u} \in \bdd{H}^s(\Omega) \cap \bdd{H}^1_D(\Omega)$ and $q \in H^{s-1}(\Omega) \cap L^2_D(\Omega)$, $s > 1$, there holds 
	\begin{align}
		\label{eq:optimal u approx}
		\inf_{ \bdd{v} \in \bdd{X}_D} \|\bdd{u} - \bdd{v}\|_{1} &\leq C h^{\min(p, s-1)} p^{-(s-1)} \|\bdd{u}\|_{s}, \\
		\label{eq:optimal q approx}
		\inf_{r \in \dive \bdd{X}_D } \|q - r\| &\leq C h^{\min(p, s-1)} p^{-(s-1)} \|q\|_{s-1},
	\end{align}
	where $C$ is independent of $\bdd{u}$, $q$, $h$, and $p$.
\end{theorem}
The conditions needed in \cref{thm:optimal approx} are quite standard, apart from the requirement that the mesh be corner-split at Dirichlet vertices. We refer to \cite[p. 35]{AinCP21LE} for a detailed characterization of the remaining mesh conditions in \cref{thm:optimal approx} and assume they hold for the remainder of this paper. Although some progress on barycenter-refined meshes \cite{Zhang05} and uniform tetrahedral grids \cite{Zhang11} have been made for the 3D Scott-Vogelius elements, their stability, approximation, and exact sequence properties remain open.

The proof of \cref{thm:optimal approx} is given in \cref{sec:exact sequence}. The inf-sup condition \cref{eq:discrete inf-sup} with $\beta$ replaced by $C p^{-K}$ for $p$ and $K > 0$ sufficiently large, was shown in \cite{Vogelius83divinv} -- the restriction on the polynomial degree was subsequently relaxed to $p \geq 4$ in \cite{ScottVog85}. Here, we show that the elements are uniformly stable in both $h$ and $p$. Even though (optimal) approximation properties of the space $\bdd{X}_D$ expressed in \cref{eq:optimal u approx} are a consequence of standard approximation theory for $hp$-finite elements \cite{Schwab98}, the result \cref{eq:optimal q approx} on the optimal approximability of the space $\dive \bdd{X}_D$ is also new, although the result was known for the pure traction problem ($|\Gamma_D| = 0$) on a fixed mesh \cite[Lemma 3.3]{Vogelius83le}. Only one other conforming finite element discretization on (again corner-split) triangular meshes is known to be uniformly inf-sup stable in $h$ and $p$ and possess optimal approximation properties \cite{AinCP19StokesI,AinCP19StokesII}.

\subsection{Exact Sequence Properties}
\label{sec:exact sequence}

Let $\{ \Gamma_{D,j} \}_{j=1}^{J}$ denote the connected components of $\Gamma_D$ and define
\begin{align*}
	H^2_D(\Omega) &:= \{ \psi \in H^2(\Omega) : \psi|_{\Gamma_{D,1}} = 0, \ \text{$\psi|_{\Gamma_{D,j}}$ is constant, $2 \leq j \leq J$, $\partial_n \psi|_{\Gamma_D} = 0$} \}.
\end{align*}
It is not difficult to see that $\vcurl H^2_D(\Omega) \subset \bdd{H}^1_D(\Omega)$ and $\dive \bdd{H}^1_D(\Omega) \subseteq L^2_D(\Omega)$, where $\vcurl \phi = (\partial_y \phi, -\partial_x \phi)^T$. In fact, the following sequence is exact \cite[Lemma 4.6.1]{Parker22} in the sense that the kernel of each operator appearing in \cref{eq:exact sequence continuous level} equals the range of the previous operator in the sequence:
\begin{align}
	\label{eq:exact sequence continuous level}
	0 \xrightarrow{ \ \ \ \subset \ \ \ } H^2_D(\Omega) \xrightarrow{\ \ \vcurl \ \ } \bdd{H}^1_D(\Omega) \xrightarrow{ \ \ \dive \ \ } L^2_D(\Omega) \xrightarrow{ \ \ \ 0 \ \ \ } 0.
\end{align}
For instance, if $\bdd{u} \in \bdd{H}^1_D(\Omega)$ is the velocity in \cref{eq:stokes continuous} so that $\dive \bdd{u} \equiv 0$, then there exists a potential $\phi \in H^2_D(\Omega)$ such that $\bdd{u} = \vcurl \phi$.

We will show that the Scott-Vogelius finite element spaces $\bdd{X}_D$ and $\dive \bdd{X}_D$ also form part of an exact sequence. To this end, define a discrete potential space by
\begin{align*}
	\Sigma_D = \Sigma \cap H^2_D(\Omega), \qquad \text{where} \quad 
	\Sigma := \{ \psi \in C^1(\bar{\Omega}) : \psi|_{K} \in \mathcal{P}_{p+1}(K) \ \forall K \in \mathcal{T} \}.
\end{align*}
As shown in \cite{AinCP21LE,ScottVog85,Vogelius83divinv}, the space $\dive \bdd{X}_D$ satisfies a constraint at certain element vertices, which may be summarized as follows. Let $\mathcal{V}_D$ denote the set of element vertices lying on the interior of $\Gamma_D$ and $\mathcal{V}_{DN}$ denote the vertices coinciding with the intersection of $\bar{\Gamma}_D$ and $\bar{\Gamma}_N$. Additionally, given $\bdd{a} \in \mathcal{V}$, let  $\mathcal{T}_{\bdd{a}}$ denote the set of elements sharing $\bdd{a}$ as a vertex, labeled as in \cref{fig:patch schema}. Then, we have the following result:
\begin{lemma}
	\label{lem:exact sequence global}
	Let $\xi(\cdot)$ be defined as in \cref{eq:xi def} and define 
	\begin{align*}
		Q &:= \bigg\{ q \in L^2(\Omega) : q|_{K} \in \mathcal{P}_{p-1}(K) \ \forall K \in \mathcal{T}, \\ &\qquad \qquad \sum_{i=1}^{|\mathcal{T}_{\bdd{a}}|} (-1)^i q|_{K_i}(\bdd{a}) = 0 \ \forall \bdd{a} \in \mathcal{V}_I : \xi(\bdd{a}) = 0 \bigg\} \notag \\
		\intertext{and}
		Q_D &:= \bigg\{ q \in Q \cap L^2_D(\Omega) : \sum_{i=1}^{|\mathcal{T}_{\bdd{a}}|} (-1)^i q|_{K_i}(\bdd{a}) = 0 \ \forall \bdd{a} \in \mathcal{V}_D : \xi(\bdd{a}) = 0 \bigg\},
	\end{align*}
	where the elements in $\mathcal{T}_{\bdd{a}}$, $\bdd{a} \in \mathcal{V}$, are labeled as in \cref{fig:patch schema}. Then, the sequence
	\begin{align}
		\label{eq:exact sequence fem level}
		0 \xrightarrow{ \ \ \ \subset \ \ \ } \Sigma_D \xrightarrow{\ \ \vcurl \ \ } \bdd{X}_D \xrightarrow{ \ \ \dive \ \ } Q_D \xrightarrow{ \ \ \ 0 \ \ \ } 0
	\end{align}
	is exact.
\end{lemma}
\begin{proof}
	In the case $|\Gamma_{D}| = |\Gamma|$, \cref{eq:exact sequence fem level} follows from the proof of Proposition 3.2 in \cite{ScottVog85}, so we assume that $|\Gamma_D| < |\Gamma|$. The inclusion $\vcurl \Sigma_D \subset \bdd{X}_D$ follows by definition, while $\dive \bdd{X}_D = Q_D$ by \cite[Theorem 4.1]{AinCP21LE}. We may argue as in \cite{MorganScott75} and the proof of \cite[Lemma 4.6.2]{Parker22} to show that
	\begin{align*}
		\dim \Sigma_D = \dim \Sigma - 7|\mathcal{V}_D| - 5|\mathcal{V}_{DN}| - (2p-7)|\mathcal{E}_D| + |\{ \bdd{a} \in \mathcal{V}_D : \xi(\bdd{a}) = 0 \}| + J - 1,
	\end{align*}
	where $\mathcal{E}_D$ is the set of element edges lying on $\Gamma_D$. Counting the constraints on the space $\bdd{X}_D$ and $Q_D$ gives
	\begin{align*}
		\dim \bdd{X}_D &= \dim \bdd{X} - 2\left\{ |\mathcal{V}_D| + |\mathcal{V}_{DN}| + (p-1)|\mathcal{E}_D| \right\} \\
		\dim Q_D &= \dim Q - |\{ \bdd{a} \in \mathcal{V}_D : \xi(\bdd{a}) = 0 \}|.
	\end{align*}
	By \cite[Lemma 6.1]{AinCP21LE}, $\dim \Sigma + \dim Q - \dim \bdd{X} = 1$, and so
	\begin{align*}
		\dim \Sigma_D + \dim Q_D - \dim \bdd{X}_D = J - 5|\mathcal{V}_D| - 3|\mathcal{V}_{DN}| + 5|\mathcal{E}_D|.
	\end{align*}
	Moreover,
	\begin{align*}
		|\mathcal{E}_D| = \sum_{j=1}^{J} |\mathcal{E}_D \cap \Gamma_{D,j}| = \sum_{j=1}^{J} \left\{ |(\mathcal{V}_D \cup \mathcal{V}_{DN}) \cap \bar{\Gamma}_{D,j}| - 1\right\} = |\mathcal{V}_D| + |\mathcal{V}_{DN}| - J,
	\end{align*}
	where we used Euler's identity on each connected component $\Gamma_{D,j}$: $|\mathcal{E}_D \cap \Gamma_{D,j}| = |(\mathcal{V}_D \cup \mathcal{V}_{DN})  \cap \bar{\Gamma}_{D,j}| - 1$. Additionally, the endpoints of each connected component $\Gamma_{D,j}$ consist of two unique vertices in $\mathcal{V}_{DN}$, and so $|\mathcal{V}_{DN}| = 2J$. Collecting results, we have $\dim \Sigma_D + \dim Q_D - \dim \bdd{X}_D = 0.$
	The exactness of \cref{eq:exact sequence fem level} now follows using standard arguments (see e.g. \cite[Proposition 3.1]{ScottVog85} or \cite[Lemma 6.1]{AinCP21LE}).
\end{proof}

\subsection{Stability and Approximation}
With an explicit characterization $Q_D =\dive \bdd{X}_D$ thanks to \cref{eq:exact sequence fem level} in hand, we now prove \cref{thm:optimal approx}. 

\begin{proof}[Proof of \cref{thm:optimal approx}]
	\Cref{eq:discrete inf-sup 2d} is an immediate consequence of \cite[Theorem 5.1]{AinCP21LE}. Let $\tilde{Q}_D := \{ r \in L^2_D(\Omega) : r|_{K} \in \mathcal{P}_{p-1}(K) \ \forall K \in \mathcal{T} \text{ $r$ is continuous at}$ 
	$\text{noncorner vertices} \}$.
	
	By \cite[Theorem 2.1]{AinCP19StokesII}, there holds
	\begin{align}
		\label{eq:tildeq optimal approx}
		\inf_{r \in \tilde{Q}_D} \|q - r\| \leq C h^{\min(p, s-1)} p^{-(s-1)} \|q\|_{s-1},
	\end{align}
	where $C$ is independent of $h$ and $p$ in the case $L^2_D(\Omega) = L^2_0(\Omega)$. Exactly the same construction in \cite[Lemmas 4.2 \& 4.3]{AinCP19StokesII} shows that \cref{eq:tildeq optimal approx} also holds in the case $L^2_D(\Omega) = L^2(\Omega)$. Since the mesh is corner-split at Dirichlet vertices, the set $	\{  \bdd{a} \in \mathcal{V}_I \cup \mathcal{V}_D : \xi(\bdd{a}) = 0 \}$
	consists of element vertices abutting an even number of elements; i.e. $|\mathcal{T}_{\bdd{a}}|$ is even (see e.g. section 4.3 of \cite{AinCP21LE}). As a result, for $r \in \tilde{Q}_D$, the condition
	\begin{align*}
		\sum_{i=1}^{|\mathcal{T}_{\bdd{a}}|} (-1)^{i} r|_{K_i}(\bdd{a}) = 0 \ \forall \bdd{a} \in \mathcal{V}_I \cup \mathcal{V}_D : \xi(\bdd{a}) = 0
	\end{align*}
	is automatically satisfied since $q$ is continuous at noncorner vertices. Consequently, $\tilde{Q}_D \subset Q_D = \dive \bdd{X}_D$, and so \cref{eq:optimal q approx} follows from \cref{eq:tildeq optimal approx}. \Cref{eq:optimal u approx} is a consequence of standard approximation theory for $hp$-finite elements; see e.g. \cite{Schwab98}.
\end{proof}

\newpage

\bibliographystyle{siamplain}
\bibliography{references}

\begin{thebibliography}{10}

\bibitem{Ain11}
{\sc M.~Ainsworth, G.~Andriamaro, and O.~Davydov}, {\em
  {B}ernstein--{B}{\'e}zier finite elements of arbitrary order and optimal
  assembly procedures}, SIAM J. Sci. Comput., 33 (2011), pp.~3087--3109,
  \url{https://doi.org/10.1137/11082539X}.

\bibitem{AinCP19StokesI}
{\sc M.~Ainsworth and C.~Parker}, {\em {M}ass conserving mixed
  \lowercase{$hp$}-{FEM} approximations to {S}tokes flow. {P}art
  \uppercase{\textup{i}}: {U}niform stability}, SIAM J. Numer. Anal., 59
  (2021), pp.~1218--1244, \url{https://doi.org/10.1137/20M1359109}.

\bibitem{AinCP19StokesII}
{\sc M.~Ainsworth and C.~Parker}, {\em {M}ass conserving mixed
  \lowercase{$hp$}-{FEM} approximations to {S}tokes flow. {P}art
  \uppercase{\textup{ii}}: {O}ptimal convergence}, SIAM J. Numer. Anal., 59
  (2021), pp.~1245--1272, \url{https://doi.org/10.1137/20M1359110}.

\bibitem{AinCP19StokesIII}
{\sc M.~Ainsworth and C.~Parker}, {\em {A} mass conserving mixed
  \lowercase{$hp$}-{FEM} scheme for {S}tokes flow. {P}art
  \uppercase{\textup{iii}}: {I}mplementation and preconditioning}, SIAM J.
  Numer. Anal., 60 (2022), pp.~1574--1606,
  \url{https://doi.org/10.1137/21M1433927}.

\bibitem{AinCP21LE}
{\sc M.~Ainsworth and C.~Parker}, {\em Unlocking the secrets of locking:
  {Finite} element analysis in planar linear elasticity}, Comput. Methods Appl.
  Mech. Engrg., 395 (2022), p.~115034,
  \url{https://doi.org/10.1016/j.cma.2022.115034}.

\bibitem{Brenner08}
{\sc S.~C. Brenner and L.~R. Scott}, {\em The Mathematical Theory of Finite
  Element Methods}, vol.~15 of Texts in Applied Mathematics, Springer-Verlag,
  New York, 3rd~ed., 2008.

\bibitem{FortinGlow83}
{\sc M.~Fortin and R.~Glowinski}, {\em Augmented Lagrangian Methods:
  Applications to the Numerical Solution of Boundary-value Problems}, Stud.
  Math. Appl. 15, North-Holland, Amsertdam, 1983.

\bibitem{Glowinski84}
{\sc R.~Glowinski}, {\em Numerical Methods for Nonlinear Variational Problems},
  Spring-Verlag, New York, 1984.

\bibitem{eigenweb}
{\sc G.~Guennebaud, B.~Jacob, et~al.}, {\em Eigen v3}.
\newblock http://eigen.tuxfamily.org, 2010.

\bibitem{John16}
{\sc V.~John}, {\em Finite element methods for incompressible flow problems},
  vol.~51 of Springer Series in Computational Mathematics, Springer, Cham,
  Switzerland, 2016, \url{https://doi.org/10.1007/978-3-319-45750-5}.

\bibitem{JohnLinkeMerdonNeilReb17}
{\sc V.~John, A.~Linke, C.~Merdon, M.~Neilan, and L.~G. Rebholz}, {\em On the
  divergence constraint in mixed finite element methods for incompressible
  flows}, SIAM Rev., 59 (2017), pp.~492--544,
  \url{https://doi.org/10.1137/15M1047696}.

\bibitem{Kovasznay48}
{\sc L.~I.~G. Kovasznay}, {\em Laminar flow behind a two-dimensional grid},
  Math. Proc. Cambridge Philos. Soc., 44 (1948), pp.~58--62,
  \url{https://doi.org/10.1017/S0305004100023999}.

\bibitem{Lai07spline}
{\sc M.-J. Lai and L.~L. Schumaker}, {\em Spline functions on triangulations},
  vol.~110, Cambridge University Press, Cambridge, 2007.

\bibitem{Moffatt64}
{\sc H.~K. Moffatt}, {\em Viscous and resistive eddies near a sharp corner}, J.
  Fluid Mech., 18 (1964), pp.~1--18,
  \url{https://doi.org/10.1017/S0022112064000015}.

\bibitem{MorganScott18}
{\sc H.~Morgan and L.~R. Scott}, {\em Towards a unified finite element method
  for the {S}tokes equations}, SIAM J. Sci. Comput., 40 (2018), pp.~A130--A141,
  \url{https://doi.org/10.1137/16M1103117}.

\bibitem{MorganScott75}
{\sc J.~Morgan and R.~Scott}, {\em A nodal basis for {$C^1$} piecewise
  polynomials of degree $n \geq 5$}, Math. Comp., 29 (1975), pp.~736--740,
  \url{https://doi.org/10.1090/S0025-5718-1975-0375740-7}.

\bibitem{Neilan15}
{\sc M.~Neilan}, {\em Discrete and conforming smooth de {R}ham complexes in
  three dimensions}, Math. Comp., 84 (2015), pp.~2059--2081,
  \url{https://doi.org/10.1090/S0025-5718-2015-02958-5}.

\bibitem{Parker22}
{\sc C.~Parker}, {\em High Order 2D Finite Element Methods with Extra
  Smoothness}, PhD thesis, Brown University, 2022,
  \url{https://repository.library.brown.edu/studio/item/bdr:mxzhbfy9/}.

\bibitem{Schwab98}
{\sc C.~Schwab}, {\em p- and hp-Finite Element Methods. Theory and Applications
  in Solid and Fluid Mechanics.}, Oxford University Press, Oxford, 1998.

\bibitem{ScottVog84}
{\sc L.~R. Scott and M.~Vogelius}, {\em Conforming finite element methods for
  incompressible and nearly incompressible continua}, in Large-Scale
  Computations in Fluid Mechanics, Part 2, Lectures in Appl. Math. 22, AMS,
  Providence, RI, 1985, pp.~221--244,
  \url{https://apps.dtic.mil/sti/citations/ADA141117}.

\bibitem{ScottVog85}
{\sc L.~R. Scott and M.~Vogelius}, {\em Norm estimates for a maximal right
  inverse of the divergence operator in spaces of piecewise polynomials}, ESAIM
  Math. Model. Numer. Anal., 19 (1985), pp.~111--143,
  \url{https://doi.org/10.1051/m2an/1985190101111}.

\bibitem{Vogelius83le}
{\sc M.~Vogelius}, {\em An analysis of the {$p$}-version of the finite element
  method for nearly incompressible materials}, Numer. Math., 41 (1983),
  pp.~39--53, \url{https://doi.org/10.1007/BF01396304}.

\bibitem{Vogelius83divinv}
{\sc M.~Vogelius}, {\em A right-inverse for the divergence operator in spaces
  of piecewise polynomials}, Numer. Math., 41 (1983), pp.~19--37,
  \url{https://doi.org/10.1007/BF01396303}.

\bibitem{Zhang05}
{\sc S.~Zhang}, {\em A new family of stable mixed finite elements for the {3D}
  {S}tokes equations}, Math. Comp., 74 (2005), pp.~543--554,
  \url{https://doi.org/10.1090/S0025-5718-04-01711-9}.

\bibitem{Zhang11}
{\sc S.~Zhang}, {\em Divergence-free finite elements on tetrahedral grids for
  {$k \geq 6$}}, Math. Comp., 80 (2011), pp.~669--695,
  \url{https://doi.org/10.1090/S0025-5718-2010-02412-3}.

\end{thebibliography}

\end{document}